\documentclass[12pt]{amsart}

\pdfoutput=1

\usepackage[text={500pt,650pt},centering]{geometry}

\usepackage{amsmath,amssymb,amsthm}
\usepackage{esint}
\usepackage{mathtools} 
\usepackage{graphicx}
\usepackage{MnSymbol}
\usepackage{mathtools} 
\usepackage[bottom]{footmisc}
\usepackage[colorlinks=true, pdfstartview=FitV, linkcolor=blue, citecolor=magenta, urlcolor=blue,pagebackref=false]{hyperref}
\usepackage{microtype}

\usepackage{bm}
\usepackage{dsfont}
\usepackage{xcolor}
\newtheorem{proposition}{Proposition}
\newtheorem{theorem}[proposition]{Theorem}
\newtheorem{lemma}[proposition]{Lemma}
\newtheorem{corollary}[proposition]{Corollary}

\theoremstyle{remark}

\theoremstyle{definition}

\numberwithin{equation}{section}
\numberwithin{proposition}{section}
\numberwithin{figure}{section}
\numberwithin{table}{section}
\newcommand{\dT}{{d}}

\newcommand{\N}{\mathbb{N}}

\newcommand{\R}{\mathbb{R}}

\newcommand{\E}{\mathbb{E}}

\newcommand{\eps}{\varepsilon}

\renewcommand{\le}{\leqslant}
\renewcommand{\ge}{\geqslant}
\renewcommand{\leq}{\leqslant}
\renewcommand{\geq}{\geqslant}

\renewcommand{\subset}{\subseteq}
\newcommand{\la}{\langle}
\newcommand{\ra}{\rangle}

\DeclareMathOperator{\sign}{sign}

\DeclareMathOperator{\supp}{supp}

\renewcommand{\tilde}{\widetilde}

\makeatletter
\newsavebox\myboxA
\newsavebox\myboxB
\newlength\mylenA
\newcommand*\mybar[2][0.75]{%
    \sbox{\myboxA}{$\m@th#2$}%
    \setbox\myboxB\null% Phantom box
    \ht\myboxB=\ht\myboxA%
    \dp\myboxB=\dp\myboxA%
    \wd\myboxB=#1\wd\myboxA% Scale phantom
    \sbox\myboxB{$\m@th\overline{\copy\myboxB}$}%  Overlined phantom
    \setlength\mylenA{\the\wd\myboxA}%   calc width diff
    \addtolength\mylenA{-\the\wd\myboxB}%
    \ifdim\wd\myboxB<\wd\myboxA%
       \rlap{\hskip 0.5\mylenA\usebox\myboxB}{\usebox\myboxA}%
    \else
        \hskip -0.5\mylenA\rlap{\usebox\myboxA}{\hskip 0.5\mylenA\usebox\myboxB}%
    \fi}
\makeatother

\newcommand{\Var}{\operatorname{Var}}

\newcommand{\e}{{\mathbb{E}}}
\newcommand{\Reals}{\mathbb{R}}
\newcommand{\Natural}{\mathbb{N}}

\newcommand{\bs}{{\mybar{\sigma}}}

\newcommand{\II}{\mathcal{I}}

\newcommand\rrac[2][r]{%
  \ifx r#1 (#2)\else
  \ifx s#1 [#2]\else
  \ifx c#1 \{#2\}\else
  \ifx v#1 |#2|\else
  \ifx a#1 \langle#2\rangle\else  
  \mathrm{Illegal~option}%
  \fi\fi\fi\fi\fi
}

\newcommand\Brac[2][r]{%
  \ifx r#1 \Big(#2\Big)\else
  \ifx s#1 \Big[#2\Big]\else
  \ifx c#1 \Big\{#2\Big\}\else
  \ifx v#1 \Big|#2\Big|\else
  \ifx a#1 \Big\langle#2\Big\rangle\else  
  \mathrm{Illegal~option}%
  \fi\fi\fi\fi\fi
}

\newcommand\brac[2][r]{%
  \ifx r#1 \big(#2\big)\else
  \ifx s#1 \big[#2\big]\else
  \ifx c#1 \big\{#2\big\}\else
   \ifx v#1 \big|#2\big|\else
  \ifx a#1 \big\langle#2\big\rangle\else    
  \mathrm{Illegal~option}%
  \fi\fi\fi\fi\fi
}

\begin{document}

\title[Strong replica symmetry in optimal Bayesian inference]{Strong replica symmetry in\\high-dimensional optimal Bayesian inference}

\author[J. Barbier]{Jean Barbier}
\address[Jean Barbier]{International Center for Theoretical Physics, Trieste, Italy.}
\email{jbarbier@ictp.it}

\author[D. Panchenko]{Dmitry Panchenko}\address[Dmitry Panchenko]{Department of Mathematics, University of Toronto, Ontario, Canada}
\email{panchenk@math.toronto.edu}

\begin{abstract}
We consider generic optimal Bayesian inference, namely, models of signal reconstruction where the posterior distribution and all hyperparameters are known. Under a standard assumption on the concentration of the free energy, we show how replica symmetry in the strong sense of concentration of all multioverlaps can be established as a consequence of the Franz-de Sanctis identities; the identities themselves in the current setting are obtained via a novel perturbation coming from exponentially distributed ``side-observations'' of the signal. Concentration of multioverlaps means that asymptotically the posterior distribution has a particularly simple structure encoded by a random probability measure (or, in the case of binary signal, a non-random probability measure). We believe that such strong control of the model should be key in the study of inference problems with underlying sparse graphical structure (error correcting codes, block models, etc) and, in particular, in the rigorous derivation of replica symmetric formulas for the free energy and mutual information in this context.
\end{abstract}
\date{}
\maketitle

\section{Introduction}

The contributions to the fields of high-dimensional (Bayesian) inference and machine learning coming from the mathematical physics of disordered systems are numerous. This is partly due to the by now well-established deep links between some of the archetypal models of these disciplines. Like the Ising model in physics, or the Sherrington-Kirkpatrick (SK) mean-field spin glass \cite{SKarticle}, a number of paradigmatic models in high-dimensional inference have emerged. Let us mention spiked matrix and tensor models \cite{baik2005phase,johnstone2001distribution,johnstone2004sparse} where a low-rank ``spike'' tensor to be recovered is hidden inside a full-rank noisy tensor. This idealised, yet very rich, probabilistic model of principal component analysis is directly connected to physics. Indeed, it is nothing else than the ``planted'' version of the SK model or, in the tensor case, the planted $p$-spin model\footnote{More precisely the symmetric spiked matrix model, also called spiked Wigner model, is the planted SK. The non-symmetric version, or spiked Wishart model, is the planted bipartite SK. Finally the symmetric spiked tensor model is equivalent to the planted $p$-spin model, and the non-symmetric version to the planted multipartite $p$-spin.}. Another important model is high-dimensional linear and generalised regression, that has applications in signal processing \cite{Donoho10112009,barbier_allerton_RLE,private}, communications \cite{barron2010sparse,barbier2015approximate,rush2015capacity} and machine learning \cite{barbier2017phase}. This is the planted version of (generalisations of) the famous ``perceptron'' model of statistical physics \cite{gardner1988optimal}. Optimal Bayesian inference models --optimal meaning that the true posterior is known-- are therefore generically equivalent to planted spin glasses, or, said differently, spin glasses living on their ``Nishimori line'', a peculiar region of the phase diagram on which deep identities force replica symmetry \cite{nishimori01,contucci2009spin}. 

These models have been solved in the sense of rigorously demonstrating the validity of ``replica symmetric formulas'' for the asymptotic mutual information (or free energy in physics terms) \cite{MezardParisi87b,MezardMontanari09} thanks to a combination of methods from spin glass physics, or information-theoretic and algorithmic techniques as in \cite{dia2016mutual,barbier_allerton_RLE,barbier2017mutual,private}. In particular two main proof schemes have emerged: a combination of the \emph{cavity method} \cite{MezardParisi87b} (or ``Aizenman-Sims-Starr scheme'' \cite{aizenman2003extended,Talagrand2011spina,Talagrand2011spinb,SKmodel}) and the canonical Guerra-Toninelli \emph{interpolation method} \cite{guerra2002thermodynamic,Guerra-2003} (inspired by a frequent use of interpolations in earlier works of Talagrand), used, e.g., to solve the spiked tensor models in \cite{2016arXiv161103888L,miolane2017fundamental,lesieur2017statistical}. See also \cite{el2018estimation,chen2019phase,chen2018phase,perry2018optimality,perry2020statistical} for related results. Another more recent proof strategy is an evolution of the interpolation method specifically tailored for optimal Bayesian inference problems, coined \emph{adaptive interpolation method} \cite{BarbierM17a,BarbierMacris2019}, and that has proven to be one of the simplest and most versatile technique for proving replica symmetric formulas in this context \cite{barbier2017phase,2017arXiv170910368B,RLEStructuredMatrices,NIPS2018_7453,NIPS2018_7584,barbier20190}. 

These two classes of models, namely spiked tensor estimation and regression, both possess an underlying dense graphical structure, with each ``spin'' interacting with all the others. Another important class of mean-field inference models are sparsely connected graphical models (or ``dilute models''). This includes sparse graphs error-correcting codes such as low-density parity check (LDPC) and generator matrix (LDGM) codes \cite{richardson2008modern} (the latter being the planted sparse $p$-spin model), planted combinatorial optimisation problems (random $K$-satisfiability, coloring, etc) \cite{krzakala2009hiding,coja2018information}, or models of community detection such as the stochastic and censored block models, see \cite{decelle2011asymptotic,abbe2017community} and references therein. The two proof schemes mentionned earlier extend to the sparse setting, with some new complications due to the additional layer of disorder of the graph; see \cite{franz2003replica,Franz-Leone-Toninelli,Panchenko-Talagrand-2004,coja2018information} for the extension of the canonical interpolation to sparse graphs and \cite{RSKsat,1RSB,finiteRSB,coja2018information} for the cavity method, as well as \cite{eric-censor-block} for the adaptive interpolation.

In all these works, the proofs are based in some way or another on the rigorous control of the order parameter of the model under consideration, generally in the form of an overlap between conditionally independent samples of the posterior (Gibbs) measure of the model (or ``replicas''), and/or between a sample and the planted ground-truth signal. Optimal Bayesian inference is an ubiquitous setting in the sense that the overlap can be shown to concentrate in the whole regime of parameters (amplitude of the noise, number of data points divided by the number of parameters to infer, etc). 
When the overlap is self-averaging, which is the case in optimal Bayesian inference \cite{barbier2019overlap}, spin glass models at high temperature \cite{Talagrand2011spina}, or ferromagnetic models \cite{chatterjee2015absence,2019arXiv190106521B}, one expects {\it replica symmetric} variational formulas for the asymptotic free energy or mutual information density, as was understood in the eighties by the physicists (and in mathematical literature in the nineties \cite{Pastur-Shcherbina-1991,Pastur-Shcherbina-Tirozzi-1994,talagrand1998sherrington}). Actually in the physics literature replica symmetry is generally the term used to precisely mean that the order parameter concentrates. This is in contrast with models where the overlap is not self-averaging, like in spin glasses at low temperature or combinatorial optimisation problems at high constraint density, which leads to more complicated formulas for the free energy computed using Parisi's replica symmetry breaking scheme \cite{Parisi-1979,Parisi-1980,Parisi-1983,MezardParisi87b,Talagrand-annals-2006, Talagrand2011spina,Talagrand2011spinb,SKmodel, SGpotts, SGvector} and the M\'ezard-Parisi ansatz \cite{MPbethe, MPZdiluted, MezardMontanari09}. 

In the present contribution, we prove that in optimal Bayesian inference replica symmetry holds in a \emph{strong} sense: all multioverlaps, namely overlaps between arbitrarily many replicas, do concentrate both with respect to the Gibbs measure and with respect to the disorder of the model. This key structural property is particularly important for dilute inference models. Indeed in densely connected models the physics is generally controlled by the usual overlap. But in sparse models, the additional source of disorder stemming from the graph (resulting in a local dependence of the cavity fields) implies that the whole series of multioverlaps matters, in particular for proving replica symmetric formulas. At a fully rigorous level, multioverlaps and the related notion of ``correlation decay'' \cite{krzakala2007gibbs} have been put under control in few situations, namely constraint satisfaction problems in \cite{talagrand2001high,montanari2006counting,RSKsat,bapst2016harnessing,coja2018charting,coja2018replica} that treat sub-regions of the phase diagram corresponding to ``high-temprature'' or ``low constraints density'' where replica symmetry holds, or ferromagnetic models in the whole phase diagram \cite{2019arXiv190106521B} (using very different techniques relying on the ferromagnetic nature of the models, and that therefore cannot be exported to the present setting). 

One important contribution where multioverlaps were studied, and that is of particular relevance for the present work, is the paper of Franz and de Sanctis \cite{FdS} (with some ideas already found in \cite{Franz-Leone-Toninelli}), where some analogues for the multioverlaps of the Ghirlanda-Guerra \cite{GG,panchenko2010ghirlanda} and Aizenman-Contucci identities \cite{ACident} for the usual overlap were derived, partly heuristically (see also follow-up works of \cite{barra2007stability,sollich2012spin}). Thanks to a new type of perturbation adapted to inference and inspired by \cite{FdS}, we manage here to prove that in optimal Bayesian inference the usual notion of replica symmetry, that is concentration of the overlap, induces strong replica symmetry, namely concentration of \emph{all} multioverlaps with respect to \emph{all} the randomness in the model, and this in the \emph{whole phase diagram} (an implication also exploited in the replica symmetric sub-region of constraint satisfaction problems in \cite{bapst2016harnessing,coja2018charting,coja2018replica}).

\section{Setting and main results}

\subsection{High-dimensional optimal Bayesian inference: base model}

Consider a ground-truth ``signal'' $\sigma^*=(\sigma_i^*)_{i\le N}$ generated probabilistically from a family $(P_N^*)_{N\ge 1}$ of \emph{factorised} (product) prior distributions that may depend on hyperparameter $(\theta^*_N)_{N\ge 1}$,
\begin{align}\label{ProductM}
\sigma^* \sim P_N^*(\,\cdot \mid \theta^*_N)=\prod_{i\le N} P_i^*(\,\cdot\mid\theta^*_N) \, ,
\end{align}
supported on $\Sigma^N$ for some bounded set $\Sigma \subset \Reals$. Data $Y=Y(\sigma^*)$ is generated conditionally on the unknown signal $\sigma^*$ and possibly an hyperparameter $\theta_{{\rm out},N}$:
\begin{equation}
Y \sim P_{{\rm out},N}(\,\cdot \mid \sigma^*,\theta_{{\rm out},N})\,.
\end{equation}
This formulation is very generic and the (real) data and hyperparameters can be vectors, tensors etc. The conditional distribution $P_{{\rm out},N}(\,\cdot \mid \sigma^*,\theta_{{\rm out},N})$ is called likelihood, or ``output channel''. In general the hyperparameters can also be random, with respective probability distributions $P_{\theta^*,N}$ and $P_{\theta_{\rm out},N}$. Of course this setting includes the case where some fixed number of  hyperparameters are set to specific values (i.e., are not random).

The inference task is to recover the signal $\sigma^*$ as accurately as possible given the data $Y$. We moreover assume that, given $N$, the hyperparameters $\theta_N:=(\theta^*_N, \theta_{{\rm out},N})$, the likelihood $P_{{\rm out},N}$ and the prior $P^*_N$ are known to the statistician that can therefore write down the correct posterior of the model, and call this setting {\it optimal} Bayesian inference.

For the sake of readability, from now on we will drop the size index $N$ for the aforementioned distributions and hyperparameters, but the reader should keep in mind that an increasing $N$ will construct an infinite famility of models.

Employing the language of statistical mechanics we define the base {\it Hamiltonian} $\mathcal{H}_{N}(\sigma)=\mathcal{H}_{N}(\sigma,Y,\theta_{\rm out})$ as the log-likelihood:
\begin{align}
\mathcal{H}_{N}(\sigma):= \ln P_{\rm out}(Y \mid \sigma,\theta_{\rm out})\,.
\end{align}
Then the (random) posterior distribution,  or ``Gibbs measure'' of the Bayesian inference model, is expressed using Bayes' formula:
\begin{align}
G_{N}(d\sigma):=\mathbb{P}(\sigma^*\in d\sigma \mid Y,\theta)&= \frac{1}{\mathcal{Z}_{N}(Y,\theta)} P^*(d\sigma \mid \theta^*)\exp\mathcal{H}_{N}(\sigma)\,. \label{post}
\end{align}
Note that we use the convention of having a $+$ sign in front of the Hamiltonian while there is usually a $-$ sign in statistical mechanics. The normalisation constant $\mathcal{Z}_{N}(Y,\theta):=P(Y \mid \theta)=\int P^*(d\sigma \mid \theta^*)P_{\rm out}(Y \mid \sigma,\theta_{\rm out})$ of the posterior is the {\it partition function} of the base inference model. This is the marginal distribution of the data and is called the ``evidence'' in Bayesian inference.

In addition of \eqref{ProductM} a second assumption required for our results to hold is the \emph{symmetry among spins}. This means that the random posterior (which is random through its dependence on $(\theta,\sigma^*,Y)$) is invariant in distribution under permutation of spins. Namely we assume that for any permutation $\rho$ of spin indices $\rho(\sigma):=(\sigma_{\rho(i)})_{i\le N}$,
\begin{align}
\mathbb{P}(\sigma^*\in d\sigma\mid Y,\theta) \,\overset{{\rm d}}{=}\, \mathbb{P}(\rho(\sigma^*)\in d\sigma\mid Y,\theta) \,.\label{spin_exchan} 
\end{align}
This assumption is standard in statistical mechanics models, but also in the analysis of inference in the high-dimensional regime, see the references in the next section. For example, at the level of the ``prior per component'' $P_i^*(\,\cdot\mid\theta^*_N)$, it means that it can be random, e.g. through its dependence in a single component of its random hyperparameter (in which case it could be written instead as $P_i^*(\,\cdot\mid\theta^*_N)=p^*(\,\cdot\mid(\theta^*_N)_i)$ for some $p^*$), but they must have same law for all $i\le N$. This is similar to having random external magnetic fields with same law in a spin model. 

Finally the {\it free entropy} (or minus {\it free energy}) is the averaged log-partition function:
\begin{align*}
F_{N}:= \ln\mathcal{Z}_{N}(Y,\theta) =\ln  \int P^*(d\sigma\mid\theta^*)\exp\mathcal{H}_{N}(\sigma)\,, \quad \text{with expectation} \quad \E F_{N}\,.
\end{align*}
The average $\E=\E_{\theta}\E_{\sigma^*\mid\theta^*}\E_{Y\mid \sigma^*,\theta_{\rm out}}$ is over the randomness of $(\theta,\sigma^*,Y)$. These are jointly called {\it quenched variables} as they are fixed by the realisation of the problem, in contrast with the dynamical variable $\sigma$ which fluctuates according to the posterior distribution. The averaged free entropy is minus the Shannon entropy of the evidence $P(Y\mid\theta)=\mathcal{Z}_{N}(Y,\theta)$, namely $-\e F_{N}=H(Y\mid\theta)$. Therefore it relates to the mutual information between the observations and the signal through
\begin{align*}
  I(\sigma^*;Y\mid \theta)= -\e F_{N}-H(Y\mid \sigma^*,\theta)\,.
\end{align*}
The conditional entropy $H(Y \mid \sigma^*,\theta)$ is often ``trivial'' to compute, while $\e F_{N}$ is not. The mutual information is one of the main information-theoretic quantities of interest as it contains the location of possible phase transitions in the inference problem, corresponding to its non-analyticities as a function of parameters of the problem such as the noise level or the amount of accessible data. It sometimes also allows to derive the optimal value of important error metrics, such as the minimum mean-square error through the I-MMSE relation \cite{GuoShamaiVerdu_IMMSE}, and therefore to establish fundamental limits to the quality of inference.

\subsection{Examples}\label{sec:examples}
Let us provide some examples of models that fall under the setting of optimal Bayesian inference described in the previous section. In these models, all assumptions we require, namely spin exchangability \eqref{spin_exchan} and, later, the concentration of a ``perturbed free'' quantified by \eqref{vNfreeEnergy}, can be verified at least in settings with $i)$ independent signal entries $(\sigma_i^*)$, $ii)$ conditionally (on the signal) independent data points, and $iii)$ hyperparameters $\theta$ with a law factorized over its components; e.g., this is often called the ``random design'' setting in the context of high-dimensional regression (as in the generalised linear model \eqref{ex:GLM}).

In the symmetric order-$p$ rank-$1$ tensor factorization problem, the data-tensor $Y=(Y_{i_1\ldots i_p})$ is generated through the observation model 
\begin{align}
Y_{i_1\ldots i_p}\sim \mathcal{N}\big(N^{\frac{1-p}{2}}\, \sigma^*_{i_1}\sigma^*_{i_2}\ldots \sigma^*_{i_p } ,1\big)\,, \qquad
1\le i_1\le i_2\le \ldots \le i_p\le N\,,\label{ex:tensorFacto}
\end{align}
i.i.d. conditionally on $\sigma^*$. The case $p=2$ is the Wigner spike model, or low-rank matrix factorization, and is one of the simplest probabilistic model for principal component analysis. 

Another model is the following generalized linear model:
\begin{align}
Y_\mu \sim p_{\rm out}\Big(\,\cdot\, \Big|\, \sum_{i\le N} \theta^{\rm out}_{\mu i}\sigma^*_{i}\Big), \qquad 1\le\mu\le M\,.\label{ex:GLM}
\end{align}
The $M$ observations are i.i.d. given $\mathbb{R}^{M\times N}\ni\theta_{\rm out}=(\theta_{\mu}^{\rm out})_{\mu\le M}$ and $\sigma^*$; this is the reason for the notation $p_{\rm out}$ instead of $P_{\rm out}$, the latter representing the full likelihood while the former is the conditional distribution of a single data point, i.e., $P_{\rm out}(\,\cdot\mid \theta_{\rm out} \sigma^*)=\otimes_{\mu\le M}\, p_{\rm out}(\,\cdot \mid \theta_{\mu}^{\rm out} \cdot \sigma^*)$. A particular simple deterministic case is 
\begin{align}
Y_\mu ={\sign}\sum_{i\le N} \theta^{\rm out}_{\mu i}\sigma^*_{i}\,, \qquad 1\le\mu\le M\,. \label{committee}
\end{align}
This model is the known as the perceptron neural network, see \cite{barbier2017phase}. Here $(\sigma^*_{i})_{i\le N}$ can be interpreted as the weights of a single neuron, and $(\theta_{\mu}^{\rm out})$ are $N$-dimensional data points used to generate the labels $(Y_\mu)$. The teacher-student scenario in which our results apply corresponds to the following: the teacher network \eqref{committee} (or \eqref{ex:GLM} in general) generates $Y$ from the data $\theta_{\rm out}$. The pairs $(Y_{\mu}, \theta_{\mu}^{\rm out})_{\mu\le M}$ are then used in order to train (i.e., learn the weights of) a student network with exactly the same architecture.

A richer example is a multi-layer version of the GLM above:
\begin{align}\label{ex:multiLayer}
\begin{cases}
X^{(L)}_{i_L} \sim p_{\rm out}^{(L)}\big(\,\cdot\, \big| \sum_{j\le N_{L-1}} \theta_{i_L j}^{(L)}X_{j}^{(L-1)}\big)\,, &1\le i_L\le N_L\,,\\
X_{i_{L-1}}^{(L-1)}\sim  p_{\rm out}^{(L-1)}\big(\,\cdot\, \big| \sum_{j\le N_{L-2}} \theta_{i_{L-1} j}^{(L-1)}X_{j}^{(L-2)}\big)\,, &1\le i_{L-1}\le N_{L-1}\,,\\
\qquad\qquad\qquad\qquad\vdots\\
X_{i_1}^{(1)}\sim  p_{\rm out}^{(1)}\big(\,\cdot \,\big| \sum_{j\le N} \theta_{i_1 j}^{(1)}\sigma^*_{j}\big)\,,  &1\le i_1\le N_1\,.  
\end{cases}
\end{align}
with an input $\sigma^*$ with independent entries. In this model $(X^{(\ell)})_{\ell=1}^{L-1}$ represent intermediate hidden variables, the visible variable $X^{(L)}= Y$ is the data, and $\theta_{\rm out}=(\theta^{(\ell)})_{\ell \le L}$ with $\theta^{(\ell)}$ is the weight matrix at the $\ell$-th layer. Note that in the single layer version \eqref{ex:GLM}, $\theta_{\rm out}$ was instead interpreted as data points and $\sigma^*$ was the weight vector to learn. Also $N_\ell=\Theta(N)$ for $\ell=1,\ldots,L$. This scaling for the variables sizes is often assumed in order not to make the inference of $\sigma^*$ from $Y$ impossible nor trivial. This multi-layer GLM has been studied by various authors \cite{manoel_multi-layer_2017,reeves_additivity_2017,fletcher_inference_2017,Gabrie:NIPS2018,DBLP:journals/corr/abs-1903-01293}. 

Another example could be another combination of complex statistical models such as, e.g., the following symmetric matrix factorization problem where the hidden low-rank representation $X$ of the matrix is itself generated from a generalized linear model over a more primitive signal $\sigma^*$:
\begin{align}\label{ex:mix}
\begin{cases}
Y_{ij}=(\lambda/N)^{1/2}\, X_{i}X_{j}+Z_{ij}\,,  &1\le i\le j\le N\,,\\
X_i \sim p_{\rm out}\big(\,\cdot\, \big| \sum_{j\le N} \theta^{\rm out}_{i j}\sigma^*_{j}\big)\,, &1\le i\le N\,. 
\end{cases}
\end{align}
Such model has recently been studied in \cite{aubin2019spiked}.

\subsection{The Ising spins case: perturbed model and multioverlaps concentration} 
The case of Ising spins $\sigma_i^*\in\Sigma=\{-1,1\}$ is simpler and we will consider it first before going to soft spins $\sigma_i^*\in[-1,1]$ in Section~\ref{sec:softspins_results}. For binary spins we can parametrise the prior in terms of ``external magnetic fields'' and write the concrete representation of the product measure \eqref{ProductM} as
\begin{equation}
P^*(\sigma^*\mid \theta^*) \sim \exp\sum_{i\le N} \theta_i^*\sigma_i^*\,,
\label{ProductMising}
\end{equation}
where $\sim$ means here equality up to a normalization constant.

\subsubsection*{Perturbed model} Computing the mutual information crucially relies on understanding the structural properties of the Gibbs measure $G_{N}$, which may be a daunting task without a bit of help. One of the most important ideas that have emerged in the study of such systems (and related spin glass models) is that one can often slightly modify the model in a way that does not affect the free entropy per variable in the thermodynamic $N\to+\infty$ limit but, at the same time, enforces ``good structural properties'' of the \emph{perturbed Gibbs measure}. This idea is not new: for example in the fully connected ferromagnetic Ising model without external magnetic field, the non-physical $0$ magnetisation solution of the mean-field free entropy potential function present below the critical temperature due to the up-down symmetry is supressed by introducing a small external magnetic field that ``selects'' a physical solution with non-vanishing magnetisation. This field is then removed after taking the thermodynamic limit, yielding the correct result for the free entropy at zero field.

In the context of spin glasses things are more subtle as (exponentially abundant) solutions to the mean-field equations are not related to such simple symmetries that can be ``broken by hand''. But yet ``good structural properties'' can be obtained thanks to perturbations (usually of the mixed $p$-spin type) that, e.g., translate into the so-called Aizenman-Contucci identities \cite{ACident} and Ghirlanda-Guerra identities \cite{GG,panchenko2010ghirlanda}, and then ultrametricity \cite{AAultra,panchenko2010connection,panchenko2013parisi}, two crucial ingredients in the proof of the free energy formula for the Sherrington-Kirkpatrick model in \cite{panchenko2013PF} (although the original proof by Talagrand \cite{Talagrand-annals-2006} found a way around this). In the context of high-dimensional Bayesian inference, an idea developed in \cite{macris2007griffith,MacrisKudekar2009,KoradaMacris_CDMA} (see also \cite{andrea2008estimating,coja2018information} for later modifications) is to add a noisy side gaussian channel with signal-to-noise ratio $\lambda_0$:
\begin{align*}
Y^{\rm gauss}\sim\mathcal{N}\big(\sqrt{\lambda_0\eps_N}\sigma^*,{\rm I}_N\big)\,, \quad \text{or equivalently}\quad Y^{\rm gauss} = \sqrt{\lambda_0\eps_N}\sigma^*+Z
\end{align*}
where $Z\sim \mathcal{N}(0,{\rm I}_N)$. This ``side-information'' modifies the posterior and results in an extra term in the Hamiltonian of the form (here $\cdot$ is the usual inner product between vectors)
\begin{equation}
\mathcal{H}^{\rm gauss}_N(\sigma,\lambda_0)=\mathcal{H}^{\rm gauss}_N(\sigma,\lambda_0,Y^{\rm gauss}(\sigma^*,Z),\eps_N):=
\lambda_0\eps_N\sigma^*\cdot \sigma + \sqrt{\lambda_0\eps_N}Z\cdot \sigma-\frac{1}{2}\lambda_0\eps_N\|\sigma\|^2\,, 
\label{OverPert}
\end{equation}
which corresponds to only keeping the $\sigma$-dependent terms in $-\frac12\|Y^{\rm gauss}-\sqrt{\lambda_0\eps_N} \sigma\|_2^2$ (note that the last term could be simplified too as $\|\sigma\|^2=N$ for binary spins, but for the soft spins case it must be included). Here the perturbation parameter $\lambda_0\in [1/2,1]$ and 
\begin{align}\label{scalingepsN}
1\ge \eps_N\to 0 \qquad \text{and} \qquad N\eps_N\to +\infty\,.  
\end{align}
The first condition implies that this Hamiltonian does not affect the free entropy per variable in the large $N$ limit and, under some assumptions on the model that we will state shortly, the usual two-replicas overlap (see definition below) concentrates on average over $\lambda_0\sim \mathcal{U}[1/2,1]$. The second condition enforces the perturbation to be ``strong enough'' to force overlap concentration.

However, our aim will be to show that one can force \emph{all} multioverlaps to concentrate. To this end, in the binary spin case $\sigma_i\in \{-1,1\}$, we introduce a novel type of side-information coming from an ``exponential channel'', namely the extra observations are drawn according an exponential distribution whose mean depends on the signal. To be precise, let $s_N$ be a positive sequence verifying
\begin{align}
 \frac{s_N}{N}\to 0 \qquad \text{and} \qquad \frac{s_N}{\sqrt{N}}\to +\infty  \,.\label{sNcond}
\end{align}
These conditions have the same purpose as the previous conditions \eqref{scalingepsN} for $\eps_N$. Given this sequence draw i.i.d. Poisson numbers $\pi_k\sim {\rm Poiss}(s_N)$ as well as i.i.d. random indices $i_{jk}\sim \mathcal{U}\{1,\ldots,N\}$ for $j\le \pi_k$ and $k\ge 1$. Denote the exponential probability density function of mean $\gamma$ as ${\rm Exp}(\gamma)$, namely the density of $X\sim {\rm Exp}(\gamma)$ is $\gamma\exp\{-\gamma x\}$ for $x\ge 0$. The side-information $Y^{\rm exp}=(Y^{\rm exp}_{jk})$ are, conditionally on $\sigma^*$, i.i.d. observations 
\begin{align}
 Y^{\rm exp}_{jk}\sim {\rm Exp}(1+\lambda_k\sigma^*_{i_{jk}}) \quad \text{for} \quad j\le \pi_k, \  k\ge 1\,. \label{Exp_channel}
\end{align}
Here $\lambda=(\lambda_k)_{k\geq 0}$ where $\lambda_k\sim  \mathcal{U}[2^{-k-1},2^{-k}]$ will be our ``averaging perturbation parameters'' and control the signal strength. Note the following scaling property of exponentially distributed random variables: if $X\sim {\rm Exp}(\gamma)$ then $X=a/\gamma$ for some $a\sim {\rm Exp}(1)$. This allows to introduce exponentially distributed i.i.d. ``noise variables'' $\xi=(\xi_{jk})$ that will sometimes be more convenient to work with than the actual observations:
\begin{align}
Y^{\rm exp}_{jk}=\frac{\xi_{jk}}{1+\lambda_k\sigma^*_{i_{jk}}} \quad \text{with}\quad  \xi_{jk}\sim{\rm Exp}(1)  \quad \text{for} \quad j\le \pi_k, \  k\ge 1\,.
\end{align}
These obervations yield another extra perturbation term in the Hamiltonian that takes the form
\begin{align}
 \mathcal{H}^{\rm exp}_N(\sigma,\lambda)=\mathcal{H}^{\rm exp}_N(\sigma,\lambda,Y^{\rm exp}(\sigma^*,\xi),\pi,(i_{jk})):=\sum_{k\ge 1}\sum_{j\le \pi_k} \Big(\ln(1+\lambda_k\sigma_{i_{jk}})-\frac{\lambda_k \xi_{jk} \sigma_{i_{jk}}}{1+\lambda_k\sigma^*_{i_{jk}}}\Big)\,, \label{exp_pertu}
\end{align}
which corresponds to the log-likelihood of the exponential observations re-expressed in terms of the signal and noise (up to irrelevant $\sigma$-independent terms that simplify with the normalisation).

Denote the set of all data and hyperparameters $\mathcal{S}_N:=\{W,\theta,\eps_N,\lambda,\pi,(i_{jk}),s_N\}$ for the perturbed inference problem, where $W:=(Y,Y^{\rm gauss},Y^{\rm exp})$ is the whole accessible data and recall $\theta:=(\theta^*,\theta_{\rm out})$.
Our proof will crucially rely on a set of important identities, called ``Nishimori identities'', that are only valid in the Bayes \emph{optimal} setting where $\mathcal{S}_N$ is assumed to be known, so that the posterior used for inference is the correct one. The notation $\E$ will be used for an average with respect to the quenched random variables $(\sigma^*,W,\theta,\pi,(i_{jk}))$ appearing in the ensuing expression, or equivalently when working with the independent noise variables, with respect to $(\sigma^*,Y,Z,\xi,\theta,\pi,(i_{jk}))$. The perturbation parameters are always considered \emph{fixed} if not explicitly averaged over using $\E_{\lambda}$. 

Together \eqref{OverPert} and \eqref{exp_pertu} result in a perturbed model with the Hamiltonian
\begin{equation}
\mathcal{H}_N^{\mathrm{pert}}(\sigma,\lambda):= \mathcal{H}_N(\sigma)+\mathcal{H}^{\rm gauss}_N(\sigma,\lambda_0) + \mathcal{H}^{\rm exp}_N(\sigma,\lambda)\,.
\label{HNpert}
\end{equation}
The associated Gibbs measure for the perturbed inference model, which is a proper Bayes optimal posterior distribution, reads
\begin{align}
G_{N}^{\rm pert}(\sigma,\lambda)=\mathbb{P}_\lambda(\sigma^*=\sigma \mid \mathcal{S}_N):= \frac{1}{\mathcal{Z}_{N}^{\rm pert}(\mathcal{S}_N)} \exp\Big\{\sum_{i \le N} \theta_i^*\sigma_i+\mathcal{H}_N^{\mathrm{pert}}(\sigma,\lambda)\Big\} \quad \text{for} \quad \sigma\in\{-1,1\}^N\,. \label{Gibbs}
\end{align}
The random perturbed posterior measure still verifies the spin symmetry, or symmetry among sites, i.e., for any permutation $\rho$ of spin indices 
\begin{align}
G_{N}^{\rm pert}(\sigma,\lambda) \,\overset{{\rm d}}{=}\, G_{N}^{\rm pert}(\rho(\sigma),\lambda) \,. \label{spinSymm}
\end{align}
We will use the notation $\sigma^\ell$, $\ell\geq 1$, for conditionally i.i.d. samples from $G_N^{\rm pert}(\,\cdot\,,\lambda)$, also called ``replicas''. As usual in statistical mechanics we denote with a bracket $\la\,\cdot\,\ra$ the average with respect to the product measure $G_N^{{\rm pert}}(\,\cdot\,,\lambda)^{\otimes \infty}$ acting on replicas, 
\begin{align}\label{def:bracket}
  \big\la A((\sigma^{\ell})_{\ell\in \mathcal{C}})\big\ra := \sum_{(\sigma^\ell)_{\ell \in \mathcal{C}}} A((\sigma^{\ell})_{\ell\in \mathcal{C}}) \prod_{\ell\in \mathcal{C}} G_N^{\rm pert}(\sigma^{\ell},\lambda) 
\end{align}
for a generic finite set of replica indices $\mathcal{C}$.
The above sum is over the hypercube $\{-1,1\}^{N\times |\mathcal{C}| }$. We will sometimes make the dependence on $\lambda$ explicit in the notation and write $\la\,\cdot\,\ra_{\lambda}$. 

Finally the average free entropy is
\begin{equation}
\E F_N^{\rm pert}(\lambda):= \E \ln \mathcal{Z}_{N}^{\rm pert}= \E \ln \sum_{\sigma\in\{-1,1\}^N}\exp\Big\{\sum_{i\le N} \theta_i^*\sigma_i+\mathcal{H}_N^{\mathrm{pert}}(\sigma,\lambda)\Big\}\,.
\label{pertFreeEnergy}
\end{equation}
It is not affected by the perturbation terms, that are smaller order (see the Appendix for a proof):
\begin{align}
   \frac1N|\mathbb{E} F_N^{\rm pert}(\lambda)-\mathbb{E}F_N| \le \frac{\eps_N}2 +\frac{6s_N}{N}\to 0\,. \label{pertIsnonPert}
\end{align}

\subsubsection*{Main results for Ising spins} The main quantities of interest are the \emph{multioverlaps}, which generalise the usual two-replicas (Edwards-Anderson) overlap order parameter in spin glasses:
\begin{equation}
R_{\ell_1,\ldots,\ell_n}:= \frac{1}{N}\sum_{i\le N}\sigma_i^{\ell_1}\ldots\sigma_i^{\ell_n}\,.\label{multioverlaps}
\end{equation}
When a single replica appears in some expression and no confusion can arise we simply denote it $$\sigma=\sigma^1\,.$$ Before discussing multioverlaps we recall the following by now classical result (proven in the next section for completeness). We use the compact notation $\e(\cdots)^2:=\e[(\cdots)^2]\ge (\e(\cdots))^2$. Let 
\begin{equation}
N v_N:=\sup\Big\{\e\big(F_N^{\rm pert}(\lambda)-\e F_N^{\rm pert}(\lambda)\big)^2 \, :\,  \lambda_k\in [2^{-k-1},2^{-k}] \ \text{for} \ k\geq 0\Big\}\,. \label{vNfreeEnergy}
\end{equation}
In situations where the same types of assumptions are assumed as in the present paper (mainly the spin exchangability \eqref{spin_exchan}, such as in all the references mentioned in Section~\ref{sec:examples}), $v_N$ can be upper bounded by a constant independent of $N$. The following holds (here it is not important that the spins are binary, only bounded suffices; in the case of soft spins the Gibbs average $\la \, \cdot\, \ra$ corresponds to the measure defined below by \eqref{bracket_soft}).
\begin{theorem}[Overlap concentration for bounded spins]\label{OverlapConc}Suppose that $\supp(P^*)\subseteq [-1,1]^N$. Let $\lambda_0\sim \mathcal{U}[1/2,1]$. There exists an absolute constant $C>0$ such that
\begin{equation}
\e_{\lambda_0}\e\big\la (R_{1,2} - \e\la R_{1,2} \ra)^2\big\ra
\leq \frac{C}{\eps_N}\Big(\frac{v_N}{N\eps_N}+\frac{1}{N}\Big)^{1/3}\,.
\label{Oconcentration}
\end{equation}
\end{theorem}
The upper bound here is uniform in $\lambda$. Denoting by $\e_\lambda$ the expectation in $\lambda$ when all $\lambda_k\sim \mathcal{U}[2^{-k-1},2^{-k}]$ for $k\geq 0$ are independent of each others and choosing, given $(v_N)$, an appropriate sequence $(\eps_N)$ verifying \eqref{scalingepsN}, we obtain
\begin{equation}
\lim_{N\to+\infty}\e_\lambda\e\big\la (R_{1,2} - \e\la R_{1,2} \ra)^2\big\ra= 0\,.
\label{Oconcentration2}
\end{equation}
This overlap concentration is forced by the perturbation term $\mathcal{H}^{\rm gauss}_N(\sigma,\lambda_0)$ and, once we have it, the concentration of all other multioverlaps will be forced by the perturbation term $\mathcal{H}^{\rm exp}_N(\sigma,\lambda)$. More precisely, this perturbation will be used to prove in Theorem \ref{ThFdSiden} below the analogue of the Franz-de Sanctis identities \cite{FdS}, and then we will use them to derive the following.
\begin{theorem}[Multioverlap concentration for binary spins]\label{MainTh}
Suppose that \eqref{Oconcentration2} holds, the prior factorises as \eqref{ProductM}, and the symmetry between spins \eqref{spinSymm} holds. Under \eqref{scalingepsN}, \eqref{sNcond} we have
\begin{equation}
\lim_{N\to+\infty}\e_{\lambda}\e\big\la (R_{1,\ldots,n}  - \e\la R_{1,\ldots,n} \ra)^2\big\ra = 0 \quad \text{for all} \quad n\geq 1\,.
\label{MOconcentration}
\end{equation}
\end{theorem}

Asymptotically, multioverlaps contain all information about finite dimensional distributions of the array $(\sigma_i^\ell)_{i\le N,\ell\ge 1}$ under the \emph{quenched Gibbs measure} $\e[G_N^{{\rm pert}}(\,\cdot\,,\lambda)^{\otimes \infty}]$. Indeed if one writes a generic joint moment of this measure
\begin{align}
\e\big\la \prod_{(i,\ell)\in \mathcal{C}} \sigma_{i}^{\ell}\big\ra\,, \label{jointMom}
\end{align}
where this time $\mathcal{C}$ is any finite set of pairs $(i,\ell)$ of spin/replica indices (with possible repetitions), then it can be re-expressed straightforwardly as a function of the multioverlaps. For example,
\begin{align}
\begin{split}
\e\big\la R_{1,2,4}(R_{2,3})^2\ldots\big\ra
&=
\e\big\la N^{-1}\sum_{i\leq N}\sigma_i^{1}\sigma_i^{2}\sigma_i^{4}
\times N^{-1}\sum_{j\leq N}\sigma_j^{2}\sigma_j^{3}
\times N^{-1}\sum_{k\leq N}\sigma_k^{2}\sigma_k^{3}\ldots\big\ra
\\
&=\e\big\la \sigma_1^{1}\sigma_1^{2}\sigma_1^{4}\sigma_2^{2}\sigma_2^{3}\sigma_3^{2}\sigma_3^{3}\ldots\big\ra+\mathcal{O}(N^{-1}),
\label{multi_are_spins}
\end{split}
\end{align}
by symmetry between sites/spins. Therefore controlling the multioverlaps gives precise structural information about the quenched Gibbs measure of the model.

Concretely, this result implies two things: firstly, a simple representation of asymptotic measures, and
% there exist a sequence $\lambda^N$  such that for any subsequential $(N_j)_{j\ge 1}$ thermodynamic limit along which the array $(\sigma_i^\ell)_{i,\ell\ge 1}$ converges in distribution and at the same time $\lambda^(N_j) \to \lambda$, there exists $\zeta_\lambda\in\Pr[-1,1]$ belonging to the set of probability distributions supported on $[-1,1]$ (and dependent on $\lambda$) such that spins $\sigma_i^\ell$ are generated by first generating an i.i.d. sequence $m_i\sim\zeta_\lambda$ and then flipping independent $\pm 1$ valued coins with expected value $m_i$ to output $(\sigma_i^\ell)_{\ell \ge 1}$. 
secondly, a strong decoupling phenomenon: in average w.r.t. the perturbation pameters $\lambda$, any finite number of spins are asymptotically independent of each others under $\E\langle\,\cdot\,\rangle$. Precise statements are given below, and proven in Section~\ref{sec:proof_coro}.
\begin{corollary}[Asymptotic spin distribution]
    Under the hypotheses of Theorem \ref{MainTh}, there exist a sequence of perturbation paramaters $\lambda^N$ such that for every subsequence $(N_j)_{j\geq1}$ along which the replicated system $(\sigma_i^\ell)_{i,\ell \geq1}$ converges in distribution and at the same time $\lambda^{N_j}$ converges with limit $\lambda$, there exists a probability measure $\zeta_\lambda$ over the set of Borel probability measures on $\{-1,1\}$ such that, for all $i\geq1$, the spin variables $(\sigma_i^\ell)_{\ell\geq1}$ converge jointly in distribution towards independent samples from $\mu_i$, where the $(\mu_i)_{i\ge 1}$ (which can be parametrized by their mean $(m_i)_{i\ge i}$, $m_i\in [-1,1]$) are i.i.d. measures distributed according to $\zeta_\lambda$.
\end{corollary}

Measure $\zeta_\lambda$ is not necessarily unique and may a priori depend on the subsequence $(N_j)_{j\geq1}$. The second consequence of Theorem~\ref{MainTh} is the following strong asymptotic independence.

\begin{corollary}[Strong asymptotic spin independence]
    Under the hypotheses of Theorem \ref{MainTh}, for every $k\geq 1$ and any collection $h_1,\dots,h_k:\{-1,1\} \to \R$ of continuous functions, we have that
    \begin{equation}
        \E_\lambda\Big|\E \Big\langle {\prod_{j=1}^k h_j(\sigma_{j})}\Big \rangle - \prod_{j=1}^k \E \Big\langle {h_j(\sigma_{j})}\Big\rangle\Big| \xrightarrow{N\to+\infty} 0\,.\label{decou}
    \end{equation}
\end{corollary}
In the above statement, by spin permutation symmetry (hypothesis \eqref{spinSymm}), we fix without loss of generality the spin indices to be the first $k$ ones but this choice is arbitrary.

It is important to realize that such strong control of the asympotic law(s) of the replicated system are possible to obtain because \emph{all} multioverlaps have been shown to concentrate with respect to \emph{all} the randomness in the model. Simply controlling the more standard two-replicas overlap would not be enough to get these corollaries in general. In particular, in models defined by \emph{sparse} graphical structures, such as sparse graph codes for error correction \cite{richardson2008modern} or community detection and sparse stochastic block models \cite{abbe2017community}, the asympotic finite marginals of the spins cannot be ``simply'' parametrized by few scalar parameters (linked to the usual overlap) as in models with a \emph{dense} underlying graphical structure (like the ones described in Section~\ref{sec:examples}). Our results may therefore end up being particularly useful in the context of these sparse models.

\subsubsection*{The Nishimory identity} Many proofs crucially rely on the Nishimori property of optimal Bayesian inference models. It is a simple consequence of the fact that sampling $(\sigma^*,W)$ according to their joint law is equivalent to first sampling the data $W=(Y,Y^{\rm gauss},Y^{\rm exp})$ according to its marginal, and then sampling $\sigma^*$ according to the conditional distribution which, in the Bayesian optimal setting, is the posterior distribution. This simple fact implies that, for any function $f$ of multiple replicas $(\sigma^\ell)$, the data and the signal $\sigma^*$, we have
\begin{equation}
\e_{\sigma^*}\E_{W \mid \sigma^*}\big\la f(\sigma^*,\sigma^2,\ldots,\sigma^n,W) \big\ra
=
\e_{W}\big\la f(\sigma^1,\sigma^2,\ldots,\sigma^n,W) \big\ra\,,
\label{Nishi_minimal}
\end{equation}
where $\e_{\sigma^*}$ is the expectation w.r.t. $\sigma^*$
 only, and $\E_{W \mid \sigma^*}$ the one w.r.t. $W$ conditional on $\sigma^*$. Recall that the bracket is the expectation with respect to the product posterior measure acting on the conditionally independent replicas. Inside expectations involving both the data (and the signal if explicitely appearing in $f$) and the (product) posterior measure, one replica can therefore be ``replaced'' by the planted signal, and vice-versa.
This key replica/signal symmetry is at the origin of the strong replica symmetry in optimal Bayesian inference.
%
% The array $\pi$ of i.i.d. Poisson numbers will play a special role. When we want to emphasize that an expectation is taken over everything except $(\pi,\lambda)$ we denote it by $\e_{|\pi,\lambda}:=\e_{\theta,\sigma^*,W|\sigma^*}$. Therefore the Nishimori identity implies
% \begin{equation}
% \e_{|\pi,\lambda}\big\la f(\sigma^*,\sigma^2,\ldots,\sigma^n,W) \big\ra
% =
% \e_{|\pi,\lambda}\big\la f(\sigma^1,\sigma^2,\ldots,\sigma^n,W) \big\ra\,,
% \label{Nishih}
% \end{equation}
Averaging over all $(\sigma^*,W,\theta,\pi,(i_{jk}))$ or equivalently $(\sigma^*,Y,Z,\xi,\theta,\pi,(i_{jk}))$,
\begin{equation}
\e\big\la f(\sigma^*,\sigma^2,\ldots,\sigma^n,W) \big\ra
=
\e\big\la f(\sigma^1,\sigma^2,\ldots,\sigma^n,W) \big\ra\,.
\label{Nishi}
\end{equation}

\subsection{Soft bounded spins: perturbed model and multioverlaps concentration}\label{sec:softspins_results}
Next, we consider a more general case of bounded spins, for certainty, $\sigma^*\in [-1,1]^N$. Note that, because in the binary spins case $\sigma^*\in\{-1,1\}^N$ we have $(\sigma_i^\ell)^p=1$ or $\sigma_i^\ell$ depending on the parity of $p\in\mathbb{N}$, then only multioverlaps \eqref{multioverlaps} with $\ell_1\neq \ldots \neq \ell_n$ appear when computing the joint moments of the quenched Gibbs measure as in \eqref{multi_are_spins}. This is why in Theorem~\ref{MainTh} the replica indices $1,\ldots,n$ that appear are different and this is sufficient. If instead $\sigma^*\in[-1,1]^N$ (or any other alphabet) then richer multioverlaps with generic replica indices $\ell_1,\ldots,\ell_n$ may appear. Equivalently the multioverlaps to control are, therefore,
\begin{equation}
R_{\ell_1,\ldots, \ell_n}^{(k_1,\ldots,k_n)}:= \frac{1}{N}\sum_{i\le N}(\sigma_i^{\ell_1})^{k_1}\ldots(\sigma_i^{\ell_n})^{k_n}
\end{equation}
with again $\ell_1\neq \ldots \neq \ell_n$, and each replica index $\ell_j$ comes with an integer power $k_j\geq 1$. Allowing same replica indices would be redundant as, e.g., $R_{1,1,2}^{(2,1,4)}$ would be the same as $R_{1,2}^{(3,4)}$.

\subsubsection*{Perturbed model} As a result of the increased richness of the multioverlaps definition, we will first need to control generalised overlaps
\begin{equation}
R_{1,2}^{(k)}:= \frac{1}{N}\sum_{i\le N} (\sigma_i^{1}\sigma_i^{2})^k \label{gene_overlap}
\end{equation}
for all $k\geq 1$. In order to do that, we add noisy side gaussian channels $$Y_{ik}^{\rm gauss} = \sqrt{\lambda_{0k}\eps_N}(\sigma_i^*)^k+Z_{ik}\quad \mbox{for all}\quad i\le N\quad \mbox{and} \quad  k\ge1$$ with $Z=(Z_{ik})$ with entries $Z_{ik}$ being i.i.d. standard gaussians. These modify the posterior and result in an extra term in the Hamiltonian:
\begin{equation}
\mathcal{H}^{\rm gauss}_N(\sigma,\lambda_{0}):=\sum_{k\geq 1}
\sum_{i\le N}\Bigl(
\lambda_{0k}\eps_N(\sigma_i^*\sigma_i)^k + \sqrt{\lambda_{0k}\eps_N}Z_{ik}(\sigma_i)^k-\frac{1}{2}\lambda_{0k}\eps_N(\sigma_i)^{2k}
\Bigr)\,.
\label{OverPertG}
\end{equation}
We will take $\lambda_0=(\lambda_{0k})_{k\ge 1}$ with $\lambda_{0k}\in [2^{-k-1},2^{-k}]$, which ensures that the above is well defined.

The analogue of the exponential channel perturbation \eqref{exp_pertu} is defined in exactly the same way as before, only a spin $\sigma_i$ is replaced by some polynomial $P_I(\sigma_i)$ and index $k$ is replaced by some multi-index $I$. Namely, let us consider multi-index $I$ consisting of an integer $m\geq 1$ and $m$ dyadic numbers 
\begin{equation}
a_{p}\in \{2^{-k}: k\geq 1\} \quad\mbox{for}\quad p\in \{0,\ldots, m-1\}\,.
\label{eqApq}
\end{equation}
This is a countable collection $\II$, so we can enumerate it by an injection $\iota\,:\, \II\to \Natural_{\ge 0}$. For $I\in \II,$ let
\begin{equation}
P_I(x):= 2^{-\iota(I)-m} \sum_{p=0}^{m-1} a_{p} x^p\,,\quad x\in[-1,1]
\label{POLdef}
\end{equation}
be a bounded polynomial of one spin, with values in $[-m 2^{-\iota(I)-m}, m2^{-\iota(I)-m}]\subset [-1/2,1/2]$. As before we introduce i.i.d. Poisson numbers $\pi=(\pi_I)_{I\in\mathcal{I}}$ with $\pi_I\sim{\rm Poiss}(s_N)$ as well as random indices $i_{jI}$ which are independently drawn from $\mathcal{U}\{1,\ldots,N\}$ for $j\le \pi_I$ and $I\in\mathcal{I}$. The exponential side observations are $$Y^{\rm exp}_{jI}=\frac{\xi_{jI}}{1+\lambda_IP_I(\sigma^*_{i_{jI}})} \quad\text{for} \quad j\le \pi_I \quad \text{and} \quad I\in\mathcal{I}$$ with i.i.d. noise $\xi=(\xi_{jI})$ with $\xi_{jI}\sim{\rm Exp}(1)$, and where each parameter $\lambda_I\in [1/2,1]$. Then instead of \eqref{exp_pertu}, and letting $\lambda_{\mathcal{I}}=(\lambda_I)_{I\in \mathcal{I}}$, these side observations yield the perturbation Hamiltonian
\begin{align}
 \mathcal{H}^{\rm exp}_N(\sigma,\lambda_{\mathcal{I}})
 :=\sum_{I\in \II}\sum_{j\le \pi_I} \Big(\ln\big(1+\lambda_IP_I(\sigma_{i_{jI}})\big)-\frac{\lambda_I \xi_{jI} P_I(\sigma_{i_{jI}})}{1+\lambda_I P_I(\sigma^*_{i_{jI}})}\Big) \label{PLdefG}\,.
\end{align}
Here $\lambda=(\lambda_0,\lambda_{\mathcal{I}})$. All together, \eqref{OverPertG} and \eqref{PLdefG} result in a perturbed inference model with Hamiltonian given by 
\begin{equation}
\mathcal{H}_N^{\mathrm{pert}}(\sigma,\lambda):= \mathcal{H}_N(\sigma)+\mathcal{H}^{\rm gauss}_N(\sigma,\lambda_0) + \mathcal{H}^{\rm exp}_N(\sigma,\lambda_{\mathcal{I}})\,.
\end{equation}
As the results and proofs in the Ising and soft bounded spins cases are well separated, we allow ourselves to use similarly to \eqref{def:bracket} the notation $\la \, \cdot\,\ra$ or $\la \, \cdot\,\ra_\lambda$ for the expectation with respect to the posterior Gibbs measure proportional to $\exp \mathcal{H}_N^{\mathrm{pert}}(\sigma,\lambda)$:
\begin{align}
G_{N}^{\rm pert}(\sigma,\lambda)=\mathbb{P}_\lambda(\sigma^*\in d\sigma \mid \mathcal{S}_N)&:= \frac{1}{\mathcal{Z}_{N}^{\rm pert}(\mathcal{S}_N)}\Big(\prod_{i\le N} P^*(d\sigma_i \mid \theta^*)\Big)  \exp\mathcal{H}_N^{\mathrm{pert}}(\sigma,\lambda) 
\end{align}
where now $\mathcal{S}_N:=\{W,\theta,\eps_N,\lambda,\pi,(i_{jI}),s_N\}$ with $W:=(Y,(Y_{ik}^{\rm gauss}),(Y_{jI}^{\rm exp}))$, and with Gibbs average
\begin{align}
 \big\la A((\sigma^{\ell})_{\ell\in \mathcal{C}})\big\ra &:= \int A((\sigma^{\ell})_{\ell\in \mathcal{C}}) \prod_{\ell\in \mathcal{C}} G_N^{\rm pert}(d\sigma^{\ell},\lambda) \,.\label{bracket_soft}
\end{align}
The integral is over the bounded domain $[-1,1]^{N\times |\mathcal{C}|}$. The measure $G_{N}^{\rm pert}(\cdot,\lambda)$ is the posterior distribution for a Bayesian optimal inference model, and therefore the Nishimori identity \eqref{Nishi_minimal} remains valid with the new definition of the Gibbs average, and also \eqref{Nishi} with an average $\e$ over $(\sigma^*,W,\theta,\pi,(i_{jI}))$ or equivalently $(\sigma^*,Y,Z,\xi,\theta,\pi,(i_{jI}))$.

Let the average free entropy of the pertubed model 
\begin{align*}
\E F_N^{\rm pert}(\lambda):=\e \ln \mathcal{Z}_{N}^{\rm pert}=\e\ln \int P^*(d\sigma\mid \theta^*)\, \exp \mathcal{H}_N^{\mathrm{pert}}(\sigma,\lambda) \,, 
\end{align*}
with, similarly as before,
\begin{equation}
Nv_N:=\sup\Big\{\e\big(F_N^{\rm pert}(\lambda)-\e F_N^{\rm pert}(\lambda)\big)^2 \, :\, \lambda_{0k}\in [2^{-k-1},2^{-k}] \ \text{for} \ k\geq 1,\, \lambda_I\in [1,2] \ \text{for} \ I\in\II\Big\}\,.
\end{equation}
We denote $\e_\lambda$ the expectation in $\lambda$ when $\lambda_{0k}\sim \mathcal{U}[2^{-k-1},2^{-k}]$ for $k\geq 1$ and all $\lambda_I\sim \mathcal{U}[1/2,1]$ for $I\in\II$. As in the binary case with \eqref{pertIsnonPert}, the perturbations are of a lower order with respect to the original Hamiltonian and, therefore, leave the free energy asymptotically invariant.

\subsubsection*{Main results for bounded spins} The following holds.
\begin{theorem}[Generalised overlap concentration for bounded spins]\label{GenOVer}
Suppose that $\supp(P^*)\subseteq [-1,1]^N$. Let $\lambda_{0k}\sim \mathcal{U}[2^{-k-1},2^{-k}]$. There exists an absolute constant $C>0$ such that
\begin{equation}
\e_{\lambda_{0k}}\e\big\la (R_{1,2}^{(k)} - \e\la R_{1,2}^{(k)} \ra)^2\big\ra
\leq \frac{C2^k}{\eps_N}\Big(\frac{v_N2^k}{N\eps_N}+\frac{1}{N}\Big)^{1/3} \quad \text{for all} \quad k\geq 1\,.
\label{OconcentrationG}
\end{equation}
This implies that, when the sequence $v_N$ does not grow too fast so that the above upper bound vanishes,
\begin{equation}
\lim_{N\to+\infty}\e_\lambda\e\big\la (R_{1,2}^{(k)} - \e\la R_{1,2}^{(k)} \ra)^2\big\ra= 0 \quad \text{for all} \quad k\geq 1\,.
\label{Oconcentration2G}
\end{equation}
\end{theorem}

The concentration of all other multioverlaps will be forced by the perturbation term \eqref{PLdefG} coming from side exponential channels. Again, more precisely, this perturbation will be used to prove in Theorem \ref{ThFdSidenG} below the analogue of the Franz-de Sanctis identities, and then we will use the identities to derive  the following.
\begin{theorem}[Multioverlap concentration for bounded spins]\label{MainThG}
Suppose that \eqref{Oconcentration2G} holds, the prior factorises as \eqref{ProductM}, and the symmetry between spins \eqref{spinSymm} holds. Under \eqref{scalingepsN}, \eqref{sNcond} we have
\begin{equation}
\lim_{N\to+\infty}\e_{\lambda} \e\big\la (R_{1,\ldots, n}^{(k_1,\ldots,k_n)}- \e\la R_{1,\ldots, n}^{(k_1,\ldots,k_n)}\ra)^2\big\ra = 0 \quad \text{for all}  \quad n\geq 1 \quad \text{and}\quad  k_1,\ldots,k_n\geq 1\,.
\label{MOconcentrationG}
\end{equation}
\end{theorem}
Again, the asymptotic meaning of this will be that there exists $\zeta_\lambda\in \Pr(\Pr[-1,1])$ such that, given i.i.d. $\mu_i\in \Pr[-1,1]$ from $\zeta_\lambda$, the spins $(\sigma_i^\ell)_{\ell\ge 1}$ are i.i.d. from $\mu_i$. 

\begin{corollary}[Asymptotic spin distribution]\label{cor:asymp_spin_dist}
    Under the hypotheses of Theorem \ref{MainThG}, there exist a sequence $\lambda^N$ such that for every subsequence $(N_j)_{j\geq1}$ along which the replicated system $(\sigma_i^\ell)_{i,\ell \geq1}$ converges in distribution and at the same time $\lambda^{N_j}$ converges to a certain limit $\lambda$, there exists a probability measure $\zeta_\lambda$ over the set of Borel probability measures on $[-1,1]$ such that, for all $i\geq1$, the spin variables $(\sigma_i^\ell)_{\ell\geq1}$ converge jointly in distribution towards independent samples from $\mu_i$, where the $(\mu_i)_{i\ge 1}$ are i.i.d. random measures distributed according to $\zeta_\lambda$.
\end{corollary}
\begin{corollary}[Asymptotic decoupling]\label{cor:asymp_indep_disthm}
   Under the hypotheses of Theorem \ref{MainThG} the decoupling \eqref{decou} holds for every $k\geq 1$ and any collection $h_1,\dots,h_k:[-1,1] \to \R$ of continuous functions.
\end{corollary}

\subsection{Outline of the paper} In the next Section \ref{ProofBin}, we will consider the case of binary spins and we will divide the proof into several subsections. We will first prove that the magnetisation concentrates by the Nishimori identity, and then recall a well-known proof of the overlap concentration based on the Nishimori identities and gaussian perturbation. After that, we will consider the case of general multioverlaps. We will start with a rigorous proof of the Franz-de Sanctis identities \cite{FdS} (in our setting) based on the exponential perturbation we introduced above. Then we will pass to the limit and rephrase everything in terms of the Aldous-Hoover representation. Finally, we will derive a consequence of the Franz-de Sanctis identities in this asymptotic language and show that it forces all multioverlaps to concentrate as long as the standard two-replicas overlap concentrates. In Section \ref{ProofSoft}, we will go over similar steps for general soft spins. There is some extra complexity in the Aldous-Hoover representation in this case, which is the reason why we present the case of binary spins first, namely, to illustrate the main ideas without unnecessary technicalities.

\section{The case of Ising spins: proof of Theorem~\ref{MainTh}}\label{ProofBin}
In this section we prove our main concentration theorem for the simpler binary spins case, whose proof already contains all the necessary ingredients for later generalisation to soft spins. In fact, in the case of soft spins almost all proofs will be identical by replacing a spin $\sigma_i$ by a generalised spin given by a polynomial $P(\sigma_i)$, so these proofs will not be repeated later on. We write these proofs for binary spins only to simplify  notation, and we note that the binary nature of spins will never really be used until Section \ref{SecAHb1} dealing with the asymptotic Aldous-Hoover representation. Only this part will be somewhat different in the case of soft spins and the corresponding modifications will then be explained. The proof will proceed by contradiction, assuming that (\ref{MOconcentration}) does not go to zero along some subsequence. Then, along a further subsubsequence, we will make sure that good properties hold in the limit (such as the Franz-de Sanctis identities) and imply that the multioverlaps must concentrate, leading to a contradiction.

We start by considering the magnetisation and the usual overlap before moving to the novelty, namely the treatment of the higher order multioverlaps.
\subsection{Magnetisation, $n=1$.}\label{sec:magn}
Proving concentration of the magnetisation $R_1:=N^{-1}\sum_{i\le N}\sigma_i$ is very simple and follows directly from the Nishimori identity. Denote $R_*:=N^{-1}\sum_{i\le N}\sigma_i^*$. Then the Nishimori identity \eqref{Nishi} implies 
$$
\e\la R_{1} \ra = \e R_*\,, \qquad \e\la R_{1}^2 \ra = \e R_*^2\,.
$$
Under the assumption \eqref{ProductM} of factorisation of the prior the entries of $\sigma^*\sim P^*$ are independent. As their variance is bounded by $1$,
\begin{align*}
&{\rm Var}(R_1)=\e\big\la (R_{1} -\e\la R_{1} \ra)^2\big\ra=\e(R_{*} -\e R_{*})^2\le 1/N\,.
\end{align*}
Taking the average over $\lambda$ of this inequality proves Theorem~\ref{MainTh} for $n=1$.

\subsection{Overlap, $n=2$: proof of Theorem~\ref{OverlapConc}}\label{sec:Overlap2Conc} The proof given here is now standard (see, e.g., \cite{BarbierMacris2019,barbier2017phase}). We directly prove the result for soft spins, with certainty $\sigma_i,\sigma_i^*\in[-1,1]$, as it makes essentially no difference. 

For this section it is convenient to introduce $\lambda_{0,N}:= \eps_N\lambda_0\in [\eps_N/2,\eps_N]$.
Let 
\begin{align}
\mathcal{H}' := \frac{d\mathcal{H}_N^{\rm gauss}(\sigma,\lambda_0)}{d\lambda_{0,N}}=\sigma\cdot \sigma^* + \frac{\sigma\cdot Z}{2\sqrt{\lambda_{0,N}}}-\frac{\|\sigma\|^2}{2} \,.\label{def_L}
\end{align}
The overlap fluctuations are upper bounded by those of $\mathcal{L}:=\mathcal{H}'/N$, which are easier to control, as
\begin{align}
\mathbb{E}\big\langle (R_{1,2} - \mathbb{E}\langle R_{1,2} \rangle)^2\big\rangle \le 4\,\mathbb{E}\big\langle (\mathcal{L} - \mathbb{E}\langle \mathcal{L}\rangle)^2\big\rangle\,.\label{remarkable}
\end{align}
A detailed derivation of this inequality can be found in the Appendix and involves only elementary algebra using the Nishimori identity and integrations by parts with respect to the gaussian noise $Z$. Recall definition \eqref{pertFreeEnergy} for the free energy. We have the following identities: for any given realisation of the quenched disorder
\begin{align}
 \frac{dF_N^{\rm pert}}{d\lambda_{0,N}}  = \langle \mathcal{H}' \rangle \,,\quad \text{and}\quad \frac{d^2F_N^{\rm pert}}{d\lambda_{0,N}^2}  &= \big\langle (\mathcal{H}'  - \langle \mathcal{H}' \rangle)^2\big\rangle-
 \frac{\langle \sigma\rangle \cdot Z}{4 \lambda_{0,N}^{3/2}} \,.\label{second-derivative}
\end{align}
The gaussian integration by part formula $\e[Zg(Z)]=\e\, g'(Z)$ for $g$ bounded and $Z\sim \mathcal{N}(0,1)$ yields
\begin{align}
\frac{1}{\sqrt{\lambda_{0,N}}}\mathbb{E}\big[\langle  \sigma \rangle\cdot Z\big] 
   =   \mathbb{E}\big\langle \|\sigma\|^2 \big\rangle - \E\|\langle \sigma \rangle\|^2 
   % =\mathbb{E}\big\langle \|\sigma\|^2 \big\rangle - \E \big[\langle \sigma \rangle\cdot \sigma^*\big]= \mathbb{E}\big\langle \|\sigma\|^2 \big\rangle - N\,\E \langle R_{1,2}\rangle
   \label{NishiTildeZ}\,.
\end{align}
% where $R_{1,*}:=N^{-1}\sigma \cdot \sigma^*$ and we used the Nishimori identity $\E\|\langle \sigma \rangle\|^2=\E [\langle \sigma \rangle\cdot \sigma^*]$. 
Averaging these identities (all domination conditions to exchange expectation and derivatives are met) and using again the Nishimori identity and gaussian integration by parts we find 
% %
\begin{align}
 \frac{d\,\e F_N^{\rm pert}}{d\lambda_{0,N}} &= \mathbb{E}\langle \mathcal{H}' \rangle =\frac{N}{2}\mathbb{E}\langle R_{1,2}\rangle\,, \quad \text{and}\quad \frac{d^2\e F_N^{\rm pert}}{d\lambda_{0,N}^2} = \mathbb{E}\big\langle (\mathcal{H}' - \langle \mathcal{H}' \rangle)^2\big\rangle
 -\frac{1}{4\lambda_{0,N}} \mathbb{E}\big\langle \|\sigma- \langle \sigma \rangle\|^2\big\rangle\,.\label{second-derivative-average}
\end{align} 
The first derivative above can also be obtained by linking the free energy and mutual information $I(\sigma^*;W\mid\theta,\lambda,\pi,(i_{jk}),\eps_N,s_N)=- \e F_N^{\rm pert}(\lambda)+C$ for some $C$ independent of $\lambda_0$, followed by a direct application of the I-MMSE relation \cite{GuoShamaiVerdu_IMMSE}. The concentration of the overlap Theorem~\ref{OverlapConc} is then a direct consequence of the following result (combined with Fubini's Theorem) and \eqref{remarkable}:
% %
\begin{proposition}[Fluctuations of $\mathcal{L}$]\label{L-concentration} Let $\lambda_0\sim\mathcal{U}[1/2,1]$. If $v_N/(N\eps_N)\to 0$ then there exists an absolute constant $C>0$ such that $$\e_{\lambda_0}\mathbb{E}\big\langle (\mathcal{L} - \mathbb{E}\langle \mathcal{L}\rangle)^2\big\rangle \le \frac{C}{\eps_N}\Big(\frac{v_N}{N\eps_N}+\frac{1}{N}\Big)^{1/3}\,.$$
\end{proposition}
\begin{proof}
The proof of this proposition is broken in two parts, using the decomposition
\begin{align*}
\mathbb{E}\big\langle (\mathcal{L} - \mathbb{E}\langle \mathcal{L}\rangle)^2\big\rangle
& = 
\mathbb{E}\big\langle (\mathcal{L} - \langle \mathcal{L}\rangle)^2\big\rangle
+ 
\mathbb{E}(\langle \mathcal{L}\rangle - \mathbb{E}\langle \mathcal{L}\rangle)^2\,.
\end{align*}
The first type of fluctuations are with respect to the posterior distribution (or ``thermal fluctuations''), while the second fluctuations are ``quenched fluctuations'' with respect to the quenched randomness. We start with the first type, and prove, for $\lambda_0\sim\mathcal{U}[1/2,1]$,
\begin{align}
\e_{\lambda_0}\mathbb{E} \big\langle (\mathcal{L} - \langle \mathcal{L}\rangle)^2 \big\rangle  \le \frac{4+\ln2}{2N\eps_N}\,.  \label{thermalL}
\end{align}
By \eqref{second-derivative-average} we have
\begin{align*}
\int_{\eps_N/2}^{\eps_N}d\lambda_{0,N}\mathbb{E}\big\langle (\mathcal{H}' - \langle \mathcal{H}' \rangle)^2\big\rangle
&\leq \int_{\eps_N/2}^{\eps_N}d\lambda_{0,N}\Big(\frac{N}{4\lambda_{0,N}}+\frac{d^2\e F_N^{\rm pert}}{d\lambda_{0,N}^2}\Big) =
\frac{d\,\e F_N^{\rm pert}}{d\lambda_{0,N}} 
\Big|_{\lambda_{0,N}=\eps_N/2}^{\lambda_{0,N}=\eps_N}+ \frac{N\ln 2}{4}\,.
\end{align*}
By \eqref{second-derivative-average} the difference of derivatives is certainly smaller in absolute value than $N$. By changing back to $\lambda_0=\lambda_{0,N}/\eps_N$ and dividing by $1/2$ to construct the average over $\lambda_0$ then by $N^2$ gives \eqref{thermalL}.

Next we prove 
\begin{align} 
  \mathbb{E}_{\lambda_0}\mathbb{E}(\langle \mathcal{L}\rangle - \mathbb{E}\langle \mathcal{L}\rangle)^2 \le \frac{C}{\eps_N}\Big(\frac{v_N}{N\eps_N}+\frac{1}{N}\Big)^{1/3} \,.
\end{align}
 The proof is based on convexity arguments, so we need first to introduce proper $\lambda_{0,N}$-convex versions of the free entropy (and its expectation). Consider the following functions of $\lambda_{0,N}$:
\begin{align}\label{new-free}
 & \tilde F(\lambda_{0,N}) := \frac1N\big(F^{\rm pert}_N(\lambda_{0,N}) -\sqrt{\lambda_{0,N}} \sum_{i\le N}\vert Z_i\vert\big)\,, \quad \mathbb{E} \tilde F(\lambda_{0,N}):=  \frac1N\e F^{\rm pert}_N(\lambda_{0,N}) - \sqrt{\lambda_{0,N}}\, \mathbb{E}\,\vert Z_1\vert\,.
\end{align}
Because of 
\eqref{second-derivative} we see that the second derivative of $\tilde F(\lambda_{0,N})$ is non-negative so that it is convex. Evidently $\e \tilde F(\lambda_{0,N})$ is convex too.
Convexity then allows to use the following standard lemma:
\begin{lemma}[A bound for convex functions]\label{lemmaConvexity}
Let $G(x)$ and $g(x)$ be convex functions. Let $\delta>0$ and define $C^{-}_\delta(x) := g'(x) - g'(x-\delta) \geq 0$ and $C^{+}_\delta(x) := g'(x+\delta) - g'(x) \geq 0$. Then
\begin{align*}
|G'(x) - g'(x)| \leq \delta^{-1} \sum_{u \in \{x-\delta,\, x,\, x+\delta\}} |G(u)-g(u)| + C^{+}_\delta(x) + C^{-}_\delta(x)\,.
\end{align*}
\end{lemma}
First, from \eqref{new-free}, and letting $A_N := N^{-1}\sum_{i\le N} (\vert  Z_i\vert -\mathbb{E}\,\vert  Z_1\vert)$, we have 
\begin{align}\label{fdiff}
 \tilde F(\lambda_{0,N}) - \e\tilde F(\lambda_{0,N}) =\frac1N \big(F^{\rm pert}_N(\lambda_{0,N}) - \e F_N^{\rm pert}(\lambda_{0,N})\big) - \sqrt{\lambda_{0,N}}  A_N\,.
\end{align} 
Second, from \eqref{second-derivative}, \eqref{second-derivative-average} we obtain for the $\lambda_{0,N}$-derivatives
\begin{align}\label{derdiff}
\tilde F'(\lambda_{0,N}) - \e \tilde F'(\lambda_{0,N}) = 
\langle \mathcal{L} \rangle-\mathbb{E}\langle \mathcal{L} \rangle - \frac{A_N}{2\sqrt{\lambda_{0,N}}} \,.
\end{align}
From \eqref{fdiff} and \eqref{derdiff} it is then easy to show that Lemma~\ref{lemmaConvexity} implies
\begin{align}\label{usable-inequ}
\vert \langle \mathcal{L}\rangle - \mathbb{E}\langle \mathcal{L}\rangle\vert&\leq 
\delta^{-1} \sum_{u\in\, \mathcal{U}}
 \Big(\frac{\vert F_N^{\rm pert}(u) - \e F_N^{\rm pert}(u) \vert}N + \vert A_N \vert \sqrt{u} \Big)+ C_\delta^+(\lambda_{0,N}) + C_\delta^-(\lambda_{0,N}) + \frac{\vert A_N\vert}{2\sqrt{\lambda_{0,N}}} 
\end{align}
where $\mathcal{U}:=\{\lambda_{0,N} -\delta,\, \lambda_{0,N},\, \lambda_{0,N}+\delta\}$ and $$C_\delta^-(\lambda_{0,N}):= \e\tilde F'(\lambda_{0,N})-\e\tilde F'(\lambda_{0,N}-\delta)\ge 0\,,\qquad C_\delta^+(\lambda_{0,N}):=\e\tilde F'(\lambda_{0,N}+\delta)-\e\tilde F'(\lambda_{0,N})\ge 0\,.$$ 
Note that $\delta$ will 
be chosen later on strictly smaller than $\eps_N/2$ so that $\lambda_{0,N} -\delta$ remains 
positive. Remark that by independence of the noise variables $\mathbb{E}A_N^2 \le 1/N$. 
We square the identity \eqref{usable-inequ} and take its expectation. Then using $(\sum_{i\le p}v_i)^2 \le p\sum_{i\le p}v_i^2$ as well as definition \eqref{vNfreeEnergy},
\begin{align}\label{intermediate}
 \frac{1}{9}\mathbb{E}(\langle \mathcal{L}\rangle - \mathbb{E}\langle \mathcal{L}\rangle)^2
 &
 \leq \, 
 \frac{3}{N\delta^2} (v_N +\eps_N+\delta)+ C_\delta^+(\lambda_{0,N})^2 + C_\delta^-(\lambda_{0,N})^2
 + \frac{1}{2N\lambda_{0,N}} \,.
\end{align}
By \eqref{second-derivative-average} and \eqref{new-free} we have
\begin{align}
|\e\tilde F'(\lambda_{0,N})|  \leq \frac 12\Big(1  +\frac{1}{\sqrt{\lambda_{0,N}}} \Big) \quad \mbox{and thus} \quad|C_\delta^\pm(\lambda_{0,N})|\le 1  +\frac{1}{\sqrt{\eps_N/2-\delta}}\,,\label{boudfprime}  
\end{align}
because $|C_\delta^\pm(\lambda_{0,N})|=|\e\tilde F'(\lambda_{0,N}\pm\delta)-\e\tilde F'(\lambda_{0,N})|$. We reach
\begin{align*}
 \int_{\eps_N/2}^{\eps_N} d\lambda_{0,N}\, \big\{C_\delta^+(\lambda_{0,N})^2 + C_\delta^-(\lambda_{0,N})^2\big\}
 &\leq 
 \Big(1  +\frac{1}{\sqrt{\eps_N/2-\delta}}\Big)
 \int_{\eps_N/2}^{\eps_N} d\lambda_{0,N}\, \big\{C_\delta^+(\lambda_{0,N}) + C_\delta^-(\lambda_{0,N})\big\}
 \nonumber \\  
&\hspace{-4cm}= \Big(1  +\frac{1}{\sqrt{\eps_N/2-\delta}}\Big)\Big[\big(\e\tilde F(\eps_N/2+\delta) -\e \tilde F(\eps_N/2-\delta)\big)+ \big(\e\tilde F(\eps_N-\delta) - \e\tilde F(\eps_N+\delta)\big)\Big]\,.
\end{align*}
%
% %
The mean value theorem and \eqref{boudfprime} 
imply $$|\e\tilde F(\lambda_{0,N}-\delta) - \e\tilde F(\lambda_{0,N}+\delta)|\le \delta\Big(1  +\frac{1}{{\sqrt{\eps_N/2-\delta}}}\Big)$$ for $\lambda_{0,N}\in[\eps_N/2,\eps_N]$. Therefore, setting $\delta  = \delta_N$ such that $1>\delta_N/ \eps_N\to 0$ and recalling $\eps_N<1$,
\begin{align*}
\int_{\eps_N/2}^{\eps_N} d\lambda_{0,N}\, \big\{C_\delta^+(\lambda_{0,N})^2 + C_\delta^-(\lambda_{0,N})^2\big\}\leq 
 2\delta_N  \Big(1 +\frac{1}{\sqrt{\eps_N/2-\delta_N}}\Big)^2\le 4\delta_N\frac{\eps_N/2-\delta_N+1}{\eps_N/2-\delta_N}\le \frac{8\delta_N}{\eps_N/2-\delta_N}\,.
\end{align*}
 Thus, integrating \eqref{intermediate} yields
\begin{align*}
\int_{\eps_N/2}^{\eps_N} d\lambda_{0,N}\ 
 \mathbb{E}(\langle \mathcal{L}\rangle - \mathbb{E}\langle \mathcal{L}\rangle)^2&\leq \frac{27\eps_N}{2N\delta_N^2}(v_N+2\eps_N) +\frac{144\delta_N}{\eps_N} +  \frac{9 \ln 2}{4N}+O\Big(\frac{\delta_N^2}{\eps_N^2}\Big)\,.
\end{align*}
Finally we optimise the bound choosing $\delta_N^3=  \Theta(\eps_N^2(v_N+\eps_N)/N)$. Then one can verify, recalling that $N\eps_N\to+\infty$, that the condition $\delta_N/\eps_N \to 0$ is indeed verified. The dominating term $\delta_N/\eps_N$ gives the result (once re-expressing the bound in terms of $\lambda_0=\lambda_{0,N}/\eps_N$). 
\end{proof}

\subsection{Multioverlaps, $n\ge 3$}
The proof of multioverlap concentration \eqref{MOconcentration} for all $n\geq 3$ is based on a new version of the Franz-de Sanctis identities from \cite{FdS}, adapted to the context of inference, and based on the exponential channel \eqref{Exp_channel} which is a novelty of the present contribution\footnote{The similarity with the Franz-de Sanctis identities from \cite{FdS} comes from the fact that we consider a Poisson number of such side-observations $Y^{\rm exp}_{jk}$, whose numbers are controlled by $(\pi_k)$ which are Poisson distributed. In their paper Franz and de Sanctis also introduce a Poisson number of perturbations but which are of the $p$-spin form, that is a canonical perturbation in spin glass literature.}.
\begin{theorem}[Franz-de Sanctis identities in inference]\label{ThFdSiden}
Assume \eqref{ProductM} and \eqref{spinSymm} hold and $s_N\leq N$. Let $i$ be a uniform index (averaged over by $\e_i$ included in $\e$ below), and define
\begin{align}
\theta_{ik}^\ell:= \ln(1+\lambda_k \sigma_i^\ell)-\lambda_k y_{ik} \sigma_i^\ell\,,\quad
y_{ik}:=\frac{\xi}{1+\lambda_k \sigma_i^*}\,,\quad\text{and}\quad
\dT_{ik}^\ell:= \frac{y_{ik}\sigma_i^\ell}{1+\lambda_k \sigma_i^*} \label{defs_FdS_0}  
\end{align}
with $\xi\sim {\rm Exp}(1)$ independently of everything else. Then, for any $k\ge 1$ and any function $f_n$ of finitely many spins on $n$ replicas and of the signal $\sigma^*$ such that $|f_n|\leq 1$,
\begin{equation}
\e_{\lambda} \left|
\e \frac{\la{f_n  \dT_{ik}^1e^{\sum_{\ell\le n} \theta_{ik}^\ell}}\ra}{\la{e^{\theta_{ik}}}\ra^n}
- \e \la{f_n }\ra
\e \frac{\la{\dT_{ik}e^{\theta_{ik}}}\ra}{\la{e^{\theta_{ik}}}\ra}
\right| \leq \Big\{\frac{2\times 10^3+2^{k+6}}{s_N}+4\times 10^4\times\Big(\frac{v_N N}{s_N^2}\Big)^{1/3}\Big\}^{1/2}.
\label{eqFdSpertfTh}
\end{equation}
\end{theorem}

Let us denote 
\begin{align}
\mathcal{H}_k':= \frac{d\mathcal{H}^{\rm exp}_N(\sigma,\lambda)}{d\lambda_k}=\sum_{j\le \pi_k} \sigma_{i_{jk}}\Big(\frac{1}{1+\lambda_k\sigma_{i_{jk}}}-\frac{\xi_{jk}}{(1+\lambda_k\sigma_{i_{jk}}^*)^2} \Big)\,, \quad\text{and} \quad \mathcal{L}_k := \frac{\mathcal{H}_k'}{s_N}\,.\label{Hkprime}
\end{align}
Define also
\begin{align}
\tilde{\mathcal{L}}_k:=\frac1{s_N}\sum_{j\le \pi_k} \frac{\sigma_{i_{jk}} \xi_{jk}}{(1+\lambda_k\sigma_{i_{jk}}^*)^2}\,.\label{Lkprime}
\end{align}
The Franz-de Sanctis identities are a corollary of the following key result.
\begin{proposition}[Fluctuations of $\mathcal{L}_k$ and $\tilde{\mathcal{L}}_k$]\label{thm:concen_der}
Recall \eqref{vNfreeEnergy}. We have
\begin{align}
  \mathbb{E}_{\lambda_k}\mathbb{E}\big\langle (\mathcal{L}_k - \mathbb{E}\langle \mathcal{L}_k\rangle)^2 \big\rangle \le \frac{2\times 10^3+2^{k+6}}{s_N} + 2\times 10^4\times (s_N^{-2}(v_N N+s_N))^{1/3}\,.
\end{align}
In the case where $v_N\le v$ for some constant $v>0$ and recalling condition \eqref{sNcond} this gives 
\begin{align}
  \mathbb{E}_{\lambda_k}\mathbb{E}\big\langle (\mathcal{L}_k - \mathbb{E}\langle \mathcal{L}_k\rangle)^2 \big\rangle \le \frac{2\times 10^3+2^{k+6}}{s_N}+4\times 10^4\times\Big(\frac{vN}{s_N^2}\Big)^{1/3}\,.
\end{align}
As a consequence, and still in the case $v_N\le v$,
\begin{align}
  \mathbb{E}_{\lambda_k}\mathbb{E}\big\langle (\tilde{\mathcal{L}}_k - \mathbb{E}\langle \tilde{\mathcal{L}}_k\rangle)^2 \big\rangle \le \frac{2\times 10^3+2^{k+6}+40}{s_N}+4\times 10^4\times\Big(\frac{vN}{s_N^2}\Big)^{1/3}\,.\label{conc_Lprime}
\end{align}
\end{proposition}

\begin{proof}
The proof mirrors the strategy used for proving Theorem~\ref{OverlapConc}, found in section~\ref{sec:Overlap2Conc}. We start with the ``thermal concentration'', namely concentration with respect to the (perturbed) posterior distribution. We will prove
  \begin{align}
    \mathbb{E}_{\lambda_k}\mathbb{E}\big\langle (
  \mathcal{H}'_k-\langle \mathcal{H}'_k\rangle)^{2}\big\rangle\le 2^{k+6}s_N\,.\label{thermal}
  \end{align}
The proof starts from the identities (exchanging expectation and derivative can be done):
\begin{align}
\frac{dF_N^{\rm pert}}{d\lambda_k}=\langle\mathcal{H}'_k\rangle\,, \qquad \frac{d\,\e F_N^{\rm pert}}{d\lambda_k}=\mathbb{E}\langle\mathcal{H}'_k\rangle\,,\label{first_der_freeEn}
\end{align}
and for the second derivative
\begin{align}
\frac{d^2F_N^{\rm pert}}{d\lambda_k^2}=\big\langle (\mathcal{H}_k'  - \langle \mathcal{H}_k' \rangle)^2\big\rangle+\langle{\mathcal{H}_k''} \rangle\,, \qquad \frac{d^2\e F_N^{\rm pert}}{d\lambda_k^2}=\mathbb{E}\big\langle (\mathcal{H}_k'  - \langle \mathcal{H}_k' \rangle)^2\big\rangle+\e\langle\mathcal{H}_k'' \rangle\,,\label{second_der}
\end{align}
where
\begin{align}
 \mathcal{H}_k'':=\frac{d^2\mathcal{H}_k}{d\lambda_k^2} =\sum_{j\le \pi_k} \Big(-\frac{1}{(1+\lambda_k\sigma_{i_{jk}})^2} +2\frac{\sigma_{i_{jk}}\sigma^*_{i_{jk}}\xi_{jk}}{(1+\lambda_k\sigma_{i_{jk}}^*)^3}\Big) \quad \text{with}\quad |\mathbb{E}\langle \mathcal{H}_k''\rangle |\le 20 s_N\,.\label{sec_der_pert}
\end{align}
We used $\lambda_k\le 1/2$, that the spins are such that $|\sigma_{i_{jk}}|\le 1$ and $\mathbb{E}\xi_{jk}=1$. Therefore from \eqref{second_der} we get 
\begin{align}
\mathbb{E}_{\lambda_k}\mathbb{E}\big\langle (\mathcal{H}_k'  - \langle \mathcal{H}_k' \rangle)^2\big\rangle\le 2^{k+1} \frac{d\,\e F^{\rm pert}_N }{d\lambda_k}\Big|_{\lambda_k=2^{-k-1}}^{2^{-k}} + 20s_N\,.\label{to_combine}
\end{align}
Note that 
\begin{align}
\Big|\frac{d\, \e F^{\rm pert}_N}{d\lambda_k}\Big|=|\mathbb{E}\langle \mathcal{H}_k'\rangle|\le 6s_N\label{der_freeEn_bounded}
\end{align}
which implies the identity \eqref{thermal} when combined with \eqref{to_combine}.

We also need to control the fluctuations with respect to the quenched randomness. We prove:
  \begin{align}
    \mathbb{E}_{\lambda_k}\mathbb{E}\big\langle (\langle \mathcal{H}_k'\rangle-\mathbb{E}\langle \mathcal{H}_k'\rangle)^{2}\big\rangle\le 2\times 10^3 \times s_N + 2\times 10^4\times (s_N^{4}(v_NN+s_N))^{1/3}\,.\label{quenched}
  \end{align}
  As before, let us define proper functions which are convex in $\lambda_k$ for a given $k$ (which are this time extensive):
  \begin{align}
  \tilde F(\lambda_k):=F^{\rm pert}_N +\sum_{j\le \pi_k} \big(8\lambda_k^2 \xi_{jk} -\ln(1-\lambda_k )\big) \quad \text{and} \quad  \mathbb{E}\tilde F(\lambda_k) = \e F_N^{\rm pert}+ s_N\big(8\lambda_k^2-\ln(1-\lambda_k)\big)\,.\label{def:tildeFs}
  \end{align}
  In particular
  \begin{align}
  \tilde F-\e\tilde F:=F_N^{\rm pert}-\e F_N^{\rm pert} +A_N\,, \quad A_N :=\sum_{j\le \pi_k} \big(8\lambda_k^2 \xi_{jk} -\ln(1-\lambda_k )\big)-s_N\big(8\lambda_k^2  -\ln(1-\lambda_k )\big)\,.\label{0.10}
  \end{align}
  From \eqref{second_der} and \eqref{sec_der_pert} one can easily see that these functions are convex in $\lambda_k$. Applying lemma~\ref{lemmaConvexity} to $\lambda_k\mapsto \tilde F$ and $\lambda_k\mapsto\e \tilde F$, and using identities \eqref{first_der_freeEn} and \eqref{0.10}, yields (we slightly abuse notation and use $\lambda_k$ for both the variable and a specific value)
\begin{align}
| \mathbb{E}\langle\mathcal{H}'_k\rangle-\langle\mathcal{H}'_k\rangle|\le \Big|\frac{dA_N}{d\lambda_k}\Big|+\delta^{-1}\sum_{u\in \,\mathcal{U}}\big(|\tilde F(u)-\e \tilde F(u)|+|A(\lambda_k=u)|\big) +C_\delta^+(\lambda_k)+C_\delta^-(\lambda_k)
\end{align}
where $\mathcal{U}:=\{\lambda_k-\delta,\lambda_k,\lambda_k+\delta\}$ and 
\begin{align}
\qquad C_\delta^-(\lambda_k):=\e \tilde F'(\lambda_k)-\e \tilde F'(\lambda_k-\delta)\ge 0\,, \qquad C_\delta^+(\lambda_k):=\e \tilde F'(\lambda_k+\delta)-\e \tilde F'(\lambda_k)\ge 0\,. \label{Cs}  
\end{align}
Here the prime symbol $'$ means $\lambda_k$-derivative. Then $(\sum_{i\le p}v_i)^2\le p\sum_{i\le p}v_i^2$ implies, when also taking the quenched expectation, that the above inequality becomes
\begin{align}
\frac19\mathbb{E}(\langle\mathcal{H}'_k\rangle-\mathbb{E}\langle\mathcal{H}'_k\rangle)^2\le \mathbb{E}\Big(\frac{dA_N}{d\lambda_k}\Big)^2 +3\delta^{-2}(Nv_N +\sup_{\lambda_k}\mathbb{E}A_N^2)+\mathbb{E}\big[C_\delta^+(\lambda_k)^2+C_\delta^-(\lambda_k)^2\big]\,.\label{ineq_squared}
\end{align}
Denote $a:=\sum_{j\le \pi_k} \big(8\lambda_k^2 \xi_{jk} -\ln(1-\lambda_k )\big)$. Using the law of total variance we start by controlling
\begin{align*}
  \mathbb{E}A_N^2={\rm Var}(a)=\mathbb{E}_{\pi_k}{\rm Var}_{(\xi_{jk})}(a)+{\rm Var}(\mathbb{E}_{(\xi_{jk})} a)=64\lambda_k^4s_N + \big(8\lambda_k^2  -\ln(1-\lambda_k )\big)^2s_N\le 13s_N
\end{align*}
using that the noise variables $\xi_{jk}$ are i.i.d. of variance $1$ and $\lambda_k\le 1/2$. For the next term we proceed similarly. Define $a':=\sum_{j\le \pi_k} \big(16\lambda_k \xi_{jk} +\frac{1}{1-\lambda_k} \big)$ which is the $\lambda_k$-derivative of $a$. Then
\begin{align*}
  \mathbb{E}\Big(\frac{dA_N}{d\lambda_k}\Big)^2={\rm Var}(a')=\mathbb{E}_{\pi_k}{\rm Var}_{(\xi_{jk})}(a')+{\rm Var}(\mathbb{E}_{(\xi_{jk})} a')=256\lambda_k^2s_N + \Big(16\lambda_k  +\frac{1}{1-\lambda_k} \Big)^2s_N\le 164s_N\,.
\end{align*}
Now consider the last term in \eqref{ineq_squared}. First note, using \eqref{der_freeEn_bounded} and definitions \eqref{def:tildeFs} and \eqref{Cs},
\begin{align}
| C^{\pm}_\delta(\lambda_k)|\le 2|\e \tilde F'(\lambda_k)|\le 2(6s_N+10s_N)=32s_N\,.
\end{align}
We will soon consider the $\lambda_k$-expectation of the inequality \eqref{ineq_squared}. For this particular term it gives (using Fubini)
\begin{align}
\mathbb{E}_{\lambda_k}\mathbb{E}\big[C_\delta^+(\lambda_k)^2+C_\delta^-(\lambda_k)^2\big]\le 32s_N\mathbb{E}\,\mathbb{E}_{\lambda_k}\big[C_\delta^+(\lambda_k)+C_\delta^-(\lambda_k)\big]\,.
\end{align}
By definition \eqref{Cs}, and as $\lambda_k\sim  \mathcal{U}[2^{-k-1},2^{-k}]$,
\begin{align*}
 \mathbb{E}_{\lambda_k}\big[C_\delta^+(\lambda_k)+C_\delta^-(\lambda_k)\big]&=\frac12\Big[\big(\e\tilde F(2^{-k}+\delta)-\e\tilde F(2^{-k}-\delta)\big)+\big(\e\tilde F f(2^{-k-1}-\delta)-\e\tilde F(2^{-k-1}+\delta)\big)\Big]\\
 &\le\frac12\times 4\times 2 \delta\times 16s_N=64\,\delta\, s_N
\end{align*}
using again $|\e\tilde F'|\le 16 s_N$. Therefore
\begin{align}
 \mathbb{E}_{\lambda_k}\mathbb{E}\big[C_\delta^+(\lambda_k)^2+C_\delta^-(\lambda_k)^2\big]\le 2^{11}\delta\,s_N^2 \,.
\end{align}
Gathering all our results in \eqref{ineq_squared} that we average over $\lambda_k$ yields
\begin{align}
\mathbb{E}_{\lambda_k}\mathbb{E}(\langle\mathcal{H}'_k\rangle-\mathbb{E}\langle\mathcal{H}'_k\rangle)^2\le 1496\,s_N +27\delta^{-2}(Nv_N+13s_N)+18432\,\delta\,s_N^2\,.\label{ineq_squared_2}
\end{align}
The bound is optimised choosing $\delta=\delta_N$ with $\delta_N^3=\Theta(s_N^{-2}(Nv_N  + s_N))$ which finally yields the inequality \eqref{quenched}. Combining the thermal \eqref{thermal} and quenched \eqref{quenched} bounds, and dividing by $s_N^2$, ends the proof of the first part of Proposition~\ref{thm:concen_der}.

In order to deduce the concentration result for $\tilde{\mathcal{L}}_k$ notice that by the Nishimori identity the first term entering ${\mathcal{L}}_k$ concentrates automatically: letting $g(\sigma,\pi_k):= \sum_{j\le \pi_k} \sigma_{i_{jk}}/(1+\lambda_k\sigma_{i_{jk}})$ we have
\begin{align}
\e\big\langle(g(\sigma)-\e\langle g(\sigma)\rangle)^2\big\rangle = \e\big(g(\sigma^*)-\e g(\sigma^*)\big)^2={\rm Var}(g(\sigma^*,\pi_k))\,.
\end{align}
By the law of total variance this gives
\begin{align}
 {\rm Var}(g(\sigma^*,\pi_k))=\e_{\pi_k}{\rm Var}_{\sigma^*}(g(\sigma^*,\pi_k))+{\rm Var}(\e_{\sigma^*} g(\sigma^*,\pi_k)) \le 4s_N+4s_N
\end{align}
using that $|\sigma^*_{i_{jk}}/(1+\lambda_k\sigma_{i_{jk}}^*)|\le 2$ because $\lambda_k\le 1/2$, and $\e \pi_k={\rm Var}(\pi_k)=s_N$. Therefore
\begin{align*}
\mathbb{E}_{\lambda_k} {\rm Var}(\tilde{\mathcal{L}}_k)&\le \mathbb{E}_{\lambda_k} {\rm Var}({\mathcal{L}}_k)+\frac{8}{s_N}+\frac2{s_N^2}\Big|{\rm Cov}\Big(g(\sigma,\pi_k),-\sum_{j\le \pi_k} \frac{\sigma_{i_{jk}}\xi_{jk}}{(1+\lambda_k\sigma_{i_{jk}}^*)^2} \Big)\Big|\\
&\le \mathbb{E}_{\lambda_k} {\rm Var}({\mathcal{L}}_k)+\frac{8}{s_N}+\frac2{s_N^2}\Big\{8s_N {\rm Var}\Big(\sum_{j\le \pi_k} \frac{\sigma_{i_{jk}}\xi_{jk}}{(1+\lambda_k\sigma_{i_{jk}}^*)^2} \Big)\Big\}^{1/2}
\end{align*}
using $|{\rm Cov}(a,b)|\le [{\rm Var}(a){\rm Var}(b)]^{1/2}$ and again ${\rm Var}(g(\sigma,\pi_k))={\rm Var}(g(\sigma^*,\pi_k))\le 8s_N$. By similar computations as before based on the law of total variance one gets that 
\begin{align*}
{\rm Var}\Big(\sum_{j\le \pi_k} \frac{\sigma_{i_{jk}}\xi_{jk}}{(1+\lambda_k\sigma_{i_{jk}}^*)^2} \Big)\le 16 \,{\rm Var}\Big(\sum_{j\le \pi_k} \xi_{jk} \Big)\le 32 s_N\,.
\end{align*}
Combining everything we reach
\begin{align*}
\mathbb{E}_{\lambda_k} {\rm Var}(\tilde{\mathcal{L}}_k)&\le \mathbb{E}_{\lambda_k} {\rm Var}({\mathcal{L}}_k)+\frac{8}{s_N}+\frac2{s_N^2}\sqrt{256 s_N^2}=\mathbb{E}_{\lambda_k} {\rm Var}({\mathcal{L}}_k)+\frac{40}{s_N}
\end{align*}
which is the result \eqref{conc_Lprime}.
\end{proof}

We are now in position to prove Theorem \ref{ThFdSiden} based on Proposition~\ref{thm:concen_der}.

\begin{proof}[Proof of Theorem \ref{ThFdSiden}]
Recall definitions \eqref{Hkprime} and \eqref{Lkprime}.  We will emphasise the dependence in the first replica $\sigma=\sigma^1$ by writting explicitly $\tilde{\mathcal{L}}_k(\sigma^1)$. By Proposition~\ref{thm:concen_der} and Cauchy-Schwarz, for any $k\ge 1$ we have
\begin{equation}
\e_{\lambda}\big|
\e \la f_n  \tilde{\mathcal{L}}_k(\sigma^1)\ra
- \e \la f_n \ra
\e \la  \tilde{\mathcal{L}}_k(\sigma)\ra  \big| \leq \Big\{\frac{2\times 10^3+2^{k+6}+40}{s_N}+4\times 10^4\times\Big(\frac{vN}{s_N^2}\Big)^{1/3}\Big\}^{1/2}\,.
\end{equation}
Recalling the definitions of the quantities entering the Franz-de Sanctis identities
\begin{align}
\theta_{ik}^\ell:= \ln(1+\lambda_k \sigma_i^\ell)-\lambda_k y_{ik} \sigma_i^\ell\,,\quad
y_{ik}:=\frac{\xi}{1+\lambda_k \sigma_i^*}\,,\quad\text{and}\quad
\dT_{ik}^\ell:= \frac{y_{ik}\sigma_i^\ell}{1+\lambda_k \sigma_i^*}\label{defs_FdS}
\end{align}
with $\xi\sim {\rm Exp}(1)$ independent of everything, it remains to show that
\begin{align}
\begin{split}\label{FdSpip}
\e \la f_n  \tilde{\mathcal{L}}_k(\sigma^1)\ra 
&=
\e\frac{\la{f_n  \dT_{ik}^1e^{\sum_{\ell\le n} \theta_{ik}^\ell}}\ra}{\la{e^{\theta_{ik}}}\ra^n} \quad \mbox{as well as} \quad \e \la \tilde{\mathcal{L}}_k(\sigma)\ra= \e \frac{\la{\dT_{ik}e^{\theta_{ik}}}\ra}{\la{e^{\theta_{ik}}}\ra}\,.
\end{split}
\end{align}
The Poisson number $\pi_k\sim{\rm Poiss}(s_N)$ appearing in $\tilde{\mathcal{L}}_k$ is independent of everything. Summing over events $\{r : \pi_k=r\geq 0 \}$ (for a given index $k\ge 1$),
\begin{align}
\e \la f_n  \tilde{\mathcal{L}}_k(\sigma^1)\ra 
&=\frac1{s_N}
\sum_{r\geq 0} \frac{s_N^r}{r!}e^{-s_N} \e \Big\langle f_n\sum_{j\le r} \frac{\sigma_{i_{jk}}^1\xi_{jk}}{(1+\lambda_k\sigma_{i_{jk}}^*)^2} \Big\rangle_{\pi_k=r} \nonumber
\\
&=
\frac1{s_N}\sum_{r\geq 1} \frac{s_N^r}{r!}e^{-s_N} \e \Big\langle f_n\sum_{j\le r}\frac{ \sigma_{i_{jk}}^1\xi_{jk}}{(1+\lambda_k\sigma_{i_{jk}}^*)^2} \Big\rangle_{\pi_k=r} \nonumber
\\
&=
\frac1{s_N}\sum_{r\geq 1} \frac{s_N^r}{r!}e^{-s_N} r\,\e \Big\langle f_n \frac{ \sigma_{i_{1k}}^1\xi_{1k}}{(1+\lambda_k\sigma_{i_{1k}}^*)^2} \Big\rangle_{\pi_k=r} \nonumber\\
&=
\sum_{r\geq 1} \frac{s_N^{r-1}}{(r-1)!}e^{-s_N} \e \Big\langle f_n\frac{ \sigma_{i_{1k}}^1\xi_{1k}}{(1+\lambda_k\sigma_{i_{1k}}^*)^2} \Big\rangle_{\pi_k=r}  \,.\label{last_poisson_manip}
\end{align}
The first equality is by definition, writting explicitely the average over the poisson number $\pi_k$, while the third equality is by symmetry under the expectation.

Recall the definition of the exponential perturbation \eqref{exp_pertu}, that can be re-written $\mathcal{H}^{\rm exp}_N(\sigma,\lambda)=\sum_{k\ge 1} \mathcal{H}_k(\sigma)$ with the obvious definition for $\mathcal{H}_k(\sigma)$. For fixed $\pi_k=r$ the latter can be decomposed as
\begin{align}
\mathcal{H}_k(\sigma^\ell)=\sum_{j\le r} \Big(\ln(1+\lambda_k\sigma^\ell_{i_{jk}})-\frac{\lambda_k \xi_{jk} \sigma^\ell_{i_{jk}}}{1+\lambda_k\sigma^*_{i_{jk}}}\Big)= \Theta_{{1k}}^\ell+ \sum_{2\le j\le r}\Theta_{{jk}}^\ell=: \Theta_{{1k}}^\ell+ \tilde{\mathcal{H}}^{r-1}_k(\sigma^\ell)\label{H_k_def}
\end{align}
where $$\Theta_{{jk}}^\ell:=\ln(1+\lambda_k\sigma^\ell_{i_{jk}})-\frac{\lambda_k \xi_{jk} \sigma^\ell_{i_{jk}}}{1+\lambda_k\sigma^*_{i_{jk}}}\,, \qquad \tilde{\mathcal{H}}_k^{r-1}(\sigma^\ell):=\sum_{2\le j\le r} \Big(\ln(1+\lambda_k\sigma^\ell_{i_{jk}})-\frac{\lambda_k \xi_{jk} \sigma^\ell_{i_{jk}}}{1+\lambda_k\sigma^*_{i_{jk}}}\Big)\,.$$
The terms $\mathcal{H}_k(\sigma^\ell)$ still appear for all replicas $\sigma^1,\ldots,\sigma^n$ in the Gibbs average $\la\,\cdot\,\ra_{\pi_k=r}$ in \eqref{last_poisson_manip}. Denote also $D^1_{{1k}}:= \sigma_{i_{1k}}^1\xi_{1k}/(1+\lambda_k\sigma_{i_{1k}}^*)^2$ and, similarly to \eqref{HNpert}, the ``partially perturbed'' Hamiltonian $$
\mathcal{H}_N^{\mathrm{g}}(\sigma):= \mathcal{H}_N(\sigma)+\mathcal{H}^{\rm gauss}_N(\sigma,\lambda_0)\,,$$ namely only perturbed by the gaussian channel. Explicitly, the term entering \eqref{last_poisson_manip} then reads after a basic manipulation (the sum $\sum_{\sigma^1\cdots \sigma^n}$ is over $\{-1,1\}^{N\times n}$)
%
% \begin{align*}
% \e\langle f_n D^1_{{1k}}\rangle_{\pi_k=r}&= \e \Big[ \prod_{2\le j\le r} \int dy_{jk}^{\rm exp} (1+\lambda_k\sigma^*_{i_{jk}})e^{-(1+\lambda_k\sigma^*_{i_{jk}})y_{jk}^{\rm exp}} \int dy_{1k}^{\rm exp} (1+\lambda_k\sigma^*_{i_{1k}})e^{-(1+\lambda_k\sigma^*_{i_{1k}})y^{\rm exp}_{1k}}\\
% %
% %
%  &\hspace{-1.5cm}\times\frac{\sum_{\sigma^1\cdots \sigma^n}  e^{\sum_{\ell\le n}(\mathcal{H}_N^{\mathrm{g}}(\sigma^\ell)+\tilde{\mathcal{H}}_k^{r-1}(\sigma^\ell))} f_n D^1_{1k} e^{\sum_{\ell \le n}\Theta_{{1k}}^\ell}}{\sum_{\sigma^1\cdots \sigma^n}  e^{\sum_{\ell\le n}(\mathcal{H}_N^{\mathrm{g}}(\sigma^\ell)+\tilde{\mathcal{H}}_k^{r-1}(\sigma^\ell))} } 
% %
%   \Big/ \frac{\sum_{\sigma^1\cdots \sigma^n}  e^{\sum_{\ell\le n}( \mathcal{H}_N^{\mathrm{g}}(\sigma^\ell)+ \tilde{\mathcal{H}}_k^{r-1}(\sigma^\ell))} e^{\sum_{\ell \le n}\Theta_{{1k}}^\ell}}{\sum_{\sigma^1\cdots \sigma^n}  e^{\sum_{\ell\le n}(\mathcal{H}_N^{\mathrm{g}}(\sigma^\ell)+\tilde{\mathcal{H}}_k^{r-1}(\sigma^\ell))} }\Big]\,.
% \end{align*}
%
\begin{align*}
\e\langle f_n D^1_{{1k}}\rangle_{\pi_k=r}&= \e  \Big(\prod_{2\le j\le r} \e_{i_{jk}} \e_{\xi_{jk}}\Big) \e_{i_{1k}}\e_{\xi_{1k}} \Big[\frac{\sum_{\sigma^1\cdots \sigma^n}  e^{\sum_{\ell\le n}(\mathcal{H}_N^{\mathrm{g}}(\sigma^\ell)+\tilde{\mathcal{H}}_k^{r-1}(\sigma^\ell))} f_n D^1_{1k} e^{\sum_{\ell \le n}\Theta_{{1k}}^\ell}}{\sum_{\sigma^1\cdots \sigma^n}  e^{\sum_{\ell\le n}(\mathcal{H}_N^{\mathrm{g}}(\sigma^\ell)+\tilde{\mathcal{H}}_k^{r-1}(\sigma^\ell))} } \\
  &\qquad\qquad\qquad\qquad\qquad\qquad\qquad\qquad \times\Big(\frac{\sum_{\sigma^1\cdots \sigma^n}  e^{\sum_{\ell\le n}( \mathcal{H}_N^{\mathrm{g}}(\sigma^\ell)+ \tilde{\mathcal{H}}_k^{r-1}(\sigma^\ell))} e^{\sum_{\ell \le n}\Theta_{{1k}}^\ell}}{\sum_{\sigma^1\cdots \sigma^n}  e^{\sum_{\ell\le n}(\mathcal{H}_N^{\mathrm{g}}(\sigma^\ell)+\tilde{\mathcal{H}}_k^{r-1}(\sigma^\ell))} }\Big)^{-1}\Big]\,.
\end{align*}
In the above identity we have explicitly written the quenched expectation with respect to the exponential noise and random spin indices indexed by $k$ and whose number is given by $\pi_k=r$. The rest of the quenched disorder is averaged all together by $\e$. We have separated the $(\Theta_{1k}^\ell)_{\ell\le n}$ explicitly. We now denote the Gibbs average with $\mathcal{H}_k(\sigma^\ell)$ replaced by $\tilde{\mathcal{H}}_k^{r-1}(\sigma^\ell)$ for all replicas as $\la\,\cdot\,\ra^{'}_{\pi_k=r}$. Then for $\pi_k=r\geq 1$ we obtain
\begin{align}
\e\langle f_n D^1_{{1k}}\rangle_{\pi_k=r}=\e\, \e_{i_{1k}} \e_{\xi_{1k}} \frac{\la{f_n  D_{1k}^1e^{\sum_{\ell \le n} \Theta^\ell_{1k}}\ra^{'}_{\pi_k=r} }}{(\la{e^{\Theta_{1k}}}\ra_{\pi_k=r}^{'})^n}\,.\label{D1_eq}
\end{align}
Note that the random variables $i_{1k}\sim\mathcal{U}\{1\,\ldots,N\}$ and $\xi_{1k}\sim {\rm Exp}(1)$ only enter the functions $D_{1k}^1$ and $(\Theta_{1k}^{\ell})_{\ell\le n}$ that have been isolated and therefore the Gibbs average $\la\,\cdot\,\ra^{'}_{\pi_k=r}$ is independent of all these. We emphasise this fact by renaming this index, noise variable and functions in the right-hand side of \eqref{D1_eq} as $i_{1k}\to i$, $\xi_{1k}\to \xi$, $\Theta_{1k}^{\ell} \to \theta_{ik}^\ell=\ln(1+\lambda_k\sigma_i^\ell)-\lambda_k  \sigma_i^\ell \xi/(1+\lambda_k\sigma_i^*)$ and $D_{1k}^1\to d_{ik}^1=\sigma_{i}^1\xi/(1+\lambda_k\sigma_{i}^*)^2$ where $\xi\sim{\rm Exp}(1)$ and $i\sim\mathcal{U}\{1,\ldots,N\}$ are independent of everything else; these match definitions \eqref{defs_FdS}. With these new variables the last equality becomes simply
$$\e\langle f_n D^1_{{1k}}\rangle_{\pi_k=r}=\e  \frac{\la{f_n  d_{ik}^1e^{\sum_{\ell \le n} \theta^\ell_{ik}}\ra^{'}_{\pi_k=r} }}{(\la{e^{\theta_{ik}}}\ra_{\pi_k=r}^{'})^n}$$
where in the right-hand side the symbol $\e$ includes the expectation with respect to the independent $\xi$ and $i$ in addition of the quenched disorder appearing in the bracket $\langle \, \cdot\,\ra_{\pi_k=r}^{'}$. Making the change of variables $m=r-1$, the sum in \eqref{last_poisson_manip} becomes
\begin{align*}
 \e \la f_n  \tilde{\mathcal{L}}_k(\sigma^1)\ra=\sum_{m\geq 0} \frac{s_N^m}{m!}e^{-s_N} \e  \frac{\la{f_n  d_{ik}^1e^{\sum_{\ell \le n} \theta^\ell_{ik}}\ra^{'}_{\pi_k=m+1} }}{(\la{e^{\theta_{ik}}}\ra_{\pi_k=m+1}^{'})^n} \,.\label{last_poisson_manip_2} 
\end{align*}
Because now in the Gibbs average $\pi_k=m+1$, the terms $\tilde{\mathcal{H}}^{m}_k(\sigma^1),\ldots,\tilde{\mathcal{H}}^{m}_k(\sigma^n)$ in $\la\,\cdot\,\ra^{'}_{\pi_k=m+1}$ become, respectively, copies of $\mathcal{H}_k(\sigma^1),\ldots,\mathcal{H}_k(\sigma^n)$ as seen from \eqref{H_k_def}. Therefore $\mathcal{H}_N^{\mathrm{g}}(\sigma^\ell)+\tilde{\mathcal{H}}_k^{m}(\sigma^\ell)$ defining the measure $\la\,\cdot\,\ra^{'}_{\pi_k=m+1}$ is equal in distribution, when $m\sim{\rm Poiss}(s_N)$, to the perturbed Hamiltonian $\mathcal{H}_N^{\mathrm{pert}}(\sigma^\ell,\lambda)$ given by \eqref{HNpert} that defines the original measure $\la\,\cdot\,\ra_{\pi_k}$ with $\pi_k\sim{\rm Poiss}(s_N)$. This proves the first equation in \eqref{FdSpip}. The second equation follows from the first replacing $f_n $ by $1$.
\end{proof}

\subsection{Passing to the limit}\label{sec:passing_limit} Suppose there exists a subsequence $(N_j)_{j\geq 1}$ along which \eqref{MOconcentration} fails for some $n\geq 3$, namely,
\begin{equation}
\e_{\lambda}\e\big\la (R_{1,\ldots,n}  - \e\la R_{1,\ldots,n} \ra)^2\big\ra \geq \delta>0\,.
\label{MOconcentrationPf}
\end{equation}
Since for a given function $f_n$ the set of its allowed arguments as well as $k\ge 1$ are countable, the equations \eqref{Oconcentration2}, \eqref{eqFdSpertfTh} and \eqref{MOconcentrationPf} imply that we can choose some $\lambda=\lambda^N=(\lambda_k^N)_{k\geq 0}$ varying with $N$, with $\lambda_k^N\in [2^{-k-1},2^{-k}]$, such that, along the same subsequence $(N_j)_{j\geq 1}$,
\begin{equation}
\e\big\la (R_{1,\ldots,n}  - \e\la R_{1,\ldots,n} \ra)^2\big\ra \geq \frac{\delta}{2}>0\,,
\label{MOconcentrationPf2}
\end{equation}
and
\begin{equation}
\e\big\la (R_{1,2}  - \e\la R_{1,2} \ra)^2\big\ra \to 0\,, \qquad
\left|
\e \frac{\la{f_n  \dT_{ik}^1e^{\sum_{\ell\le n} \theta_{ik}^\ell}}\ra}{\la{e^{\theta_{ik}}}\ra^n}
- \e \la{f_n }\ra
\e \frac{\la{\dT_{ik}e^{  \theta_{ik}}}\ra}{\la{e^{ \theta_{ik}}}\ra}
\right|\to 0
\label{eqFdSpertf32}
\end{equation} 
with definitions \eqref{defs_FdS_0}, and this jointly for all possible arguments of $f_n$ and $k\ge 1$, with these specific parameters $\lambda=\lambda^N$. The Gibbs measure $\la\,\cdot\,\ra$ is also for the Hamiltonian with these parameters $\lambda^N$.

Let us prove the existence of such $\lambda^N$. Let the multioverlap variance ${\rm Var}(R_{1,\ldots,n}):=\e\la (R_{1,\ldots,n}  - \e\la R_{1,\ldots,n} \ra)^2\ra\le 1$. We have ($\boldsymbol{1}(\cdot)$ is the indicator):
\begin{align*}
&\e_{\lambda}\boldsymbol{1}({\rm Var}(R_{1,\ldots,n})\ge \delta/2)+\delta/2 \nonumber\\
&\qquad\ge\e_{\lambda}\boldsymbol{1}({\rm Var}(R_{1,\ldots,n})\ge \delta/2)+(\delta/2)\e_{\lambda} \boldsymbol{1}({\rm Var}(R_{1,\ldots,n})< \delta/2) \nonumber\\
&\qquad\qquad\ge\e_{\lambda}[{\rm Var}(R_{1,\ldots,n})\boldsymbol{1}({\rm Var}(R_{1,\ldots,n})\ge \delta/2)]+\e_{\lambda}[{\rm Var}(R_{1,\ldots,n})\, \boldsymbol{1}({\rm Var}(R_{1,\ldots,n})< \delta/2)] \geq \delta
\end{align*}
by \eqref{MOconcentrationPf} for the last inequality, so that $$\mathbb{P}_\lambda({\rm Var}(R_{1,\ldots,n})\ge \delta/2)\ge \delta/2>0$$ uniformly in $(N_j)_{j\ge 1}$. The pairs $(f_n ,k)$ (where by $f_n$ we mean the function $f_n$ with a given set of arguments) can be injectively indexed by integers $j\ge 1$. Then denote ${\rm FdS}_j$ a bounded term of the form of what is appearing between the absolute values in \eqref{eqFdSpertfTh} for some specific $(f_n ,k)$. Then let $$|{\rm FdS}|:=\sum_{j\ge 1}2^{-j}\,|{\rm FdS}_j|\,.$$ From \eqref{Oconcentration2}, \eqref{eqFdSpertfTh} the Markov inequality implies
\begin{align*}
\mathbb{P}_\lambda({\rm Var}(R_{1,2})\le \eps)\ge 1-\frac{\e_\lambda{\rm Var}(R_{1,2})}{\eps}\,, \qquad \mathbb{P}_\lambda(|{\rm FdS}|\le \eps)\ge 1-\frac{\e_\lambda|{\rm FdS}|}{\eps}\,.
\end{align*}
As long as
\begin{align*}
  \mathbb{P}_\lambda({\rm Var}(R_{1,\ldots,n})&\ge \delta/2)+\mathbb{P}_\lambda({\rm Var}(R_{1,2})\le \eps)+\mathbb{P}_\lambda(|{\rm FdS}|\le \eps)\ge \frac{\delta}{2}+2-\frac{\e_\lambda{\rm Var}(R_{1,2})}{\eps}-\frac{\e_\lambda|{\rm FdS}|}{\eps}
\end{align*}
is also greater or equal than $2+\delta/4$, namely whenever $\eps\ge 4 (\e_\lambda{\rm Var}(R_{1,2})+\e_\lambda|{\rm FdS}|)/\delta=: C^N$ (with $C^N\to 0$ by \eqref{Oconcentration2}, \eqref{eqFdSpertfTh} when appropriately choosing $s_N$), then the following condition is satisfied
\begin{align}
 \mathbb{P}_\lambda\big(\{{\rm Var}(R_{1,\ldots,n})\ge \delta/2 \}  \cap\{{\rm Var}(R_{1,2}) \le \eps\}\cap \{|{\rm FdS}|\le \eps\}\big)\ge \frac{\delta}{4}\,.
\end{align}
Therefore, choosing an appropriate sequence $\eps=\eps^N\to 0$ along $(N_j)_{j\ge 1}$, with $\eps^N\ge C^N$, proves the existence of $\lambda^N=(\lambda_k^N)_{k\geq 0}$. 

By the Nishimori identity applied to the functions
$$ (\sigma^*,\sigma^1,\sigma^2,\ldots,\sigma^n,W)\mapsto \frac{{f_n  \dT_{ik}^1e^{\sum_{\ell\le n} \theta_{ik}^\ell}}}{\la{e^{\theta_{ik}}}\ra^n}\,, \quad (\sigma^*,\sigma^1,\sigma^2,\ldots,\sigma^n)\mapsto {f_n } \quad \text{and} \quad  (\sigma^*,\sigma^1,W)\mapsto \frac{{\dT_{ik}e^{\theta_{ik}}}}{\la{e^{\theta_{ik}}}\ra}\,,$$
we can replace the quenched signal $\sigma^*$ in all the integrands in the second equation in \eqref{eqFdSpertf32} by another replica (note that the denominators $\la{e^{\theta_{ik}}}\ra$ are just functions of the data $W$ and therefore remain unchanged by the application of the Nishimori identity). For convenience of notation, we will denote this new replica by $\sigma^\diamond$ to distinguish from the disorder $\sigma^*$ and at the same time not to occupy any specific index. Then, the second equation in \eqref{eqFdSpertf32} can be written as
\begin{equation}
\left |
\e \e_\diamond
\frac{\rrac[a]{f_n \dT_{ik}^1 e^{\sum_{\ell\le n} \theta_{ik}^\ell}}}{\rrac[a]{e^{\theta_{ik}}}^n}
- \e \e_\diamond \rrac[a]{f_n}
\e \e_\diamond\frac{\rrac[a]{\dT_{ik} e^{\theta_{ik}}}}{\rrac[a]{e^{\theta_{ik} }}}
\right|\to 0\,,
\label{eqFdSpertf3cp}
\end{equation}
where $\e_\diamond$ denotes the Gibbs average $\la\,\cdot\,\ra$ with respect to the replica $\sigma^\diamond$ only, $\la\,\cdot\,\ra$ denotes the Gibbs average with respect to all other ``standard'' replicas, $f_n$ is a function of finitely many spins on $n$ replicas and on $\sigma^\diamond$, $\xi\sim \mathrm{Exp}(1)$, and
$$
\theta_{ik}^\ell=\ln(1+\lambda_k \sigma_i^\ell)-\lambda_k y_{ik} \sigma_i^\ell\,,\quad
y_{ik}=\frac{\xi}{1+\lambda_k \sigma_i^\diamond},\quad
\dT_{ik}^\ell= \frac{y_{ik}\sigma_i^\ell}{1+\lambda_k \sigma_i^\diamond}\,,
$$
with $\lambda_k=\lambda_k^N$ for all $k\geq 0$.

Then we extract a further subsequence $(N_{j_a})_{a\ge 1}$ of $(N_j)_{j\ge 1}$ along which $\lambda^N_k\to\lambda_k\in[2^{-k-1},2^{-k}]$ for all $k\ge 0$ (by Cantor's diagonalisation). Below, we will work with the set $\Lambda:=\{\lambda_k: k\geq 1\}.$

Finally, we can choose a further subsequence along which the distribution of the array $(\sigma_i^\ell)$ under the quenched Gibbs measure $\e[G_N^{{\rm pert}}(\,\cdot\,,\lambda)^{\otimes \infty}]$ converges in the sense of finite dimensional distributions (by Prohorov's Theorem, since the space $\{-1,1\}^{\mathbb{N}^2}$ equipped with its product topology and the discrete metric is compact).

\subsection{Aldous-Hoover representation in the limit.} \label{SecAHb1}
In this latter subsequential thermodynamic limit along which the quenched Gibbs measure of $(\sigma_i^\ell)$ converges, the distribution of spins will inherit the symmetry between sites and replicas from the model,
\begin{equation}
(\sigma_i^\ell)_{i,\ell\geq 1}\stackrel{\rm d}{=}\big(\sigma_{\rho_1(i)}^{\rho_2(\ell)}\big)_{i,\ell\geq 1}
\end{equation}
for any permutations $\rho_1$ and $\rho_2$ of finitely many indices. By the Aldous-Hoover representation \cite{Aldous, Hoover} (see also Section 1.4 in \cite{SKmodel}), such symmetry implies that
\begin{equation}
(\sigma_i^\ell)_{i,\ell\ge 1} \stackrel{\rm d}{=}\big(\sigma(w,u_\ell,v_i,x_{i,\ell})\big)_{i,\ell\ge 1}
\label{AHrepr}
\end{equation}
for some function $\sigma\colon [0,1]^4\to \{-1,1\}$ (that may a priori depend on the form of $f_n$ and the subsequential limit selected in the previous section in case there are mutliple subsequential limits for the quenched Gibbs measure), and where $w$, $(u_\ell)$, $(v_i)$ and $(x_{i,\ell})$ are i.i.d. uniform $\mathcal{U}[0,1]$ random variables. This means that, along the above subsequence, for any finite subset $\mathcal{C}\subseteq \Natural^2$,
$$
\e \big\la \prod_{(i,\ell)\in \mathcal{C}} \sigma_i^\ell \big\ra\to 
\e  \prod_{(i,\ell)\in \mathcal{C}} \sigma(w,u_\ell,v_i,x_{i,\ell})\,,
$$
where in the limit, $\la\,\cdot\,\ra$ becomes the expectation in the random variables that depend on the replica indices, namely, the expectation in $(u_\ell)$ and $(x_{i,\ell})$ (see Appendix). Moreover, given representation \eqref{AHrepr}, if we denote
\begin{equation}\label{def_bs}
\bs(w,u,v):=\int_0^1\! \sigma(w,u,v,x)\, dx
\end{equation}
then the asymptotic analogue of the multioverlaps is given by
\begin{equation}
R_{\ell_1,\ldots, \ell_n} \to R^\infty_{\ell_1,\ldots, \ell_n}(w,(u_{\ell_j})_{j\le n}) :=\int_0^1 \prod_{j\leq n} \bs(w,u_{\ell_j},v)\,dv\,, \label{asymptMultiover}
\end{equation}
in the weak convergence sense, namely, the joint moments of all multioverlaps before the limit converge to joint moments of these analogues in the limit (see \cite{Pspins} or the Appendix). With this notation, in the limit along the above subsequence, the equations \eqref{MOconcentrationPf2} and \eqref{eqFdSpertf32} imply that
\begin{equation}
\e\big\la (R_{1,\ldots,n}^\infty  - \e\la R_{1,\ldots,n}^\infty  \ra)^2\big\ra \geq \frac{\delta}{2}>0\,,
\label{MOconcentrationPf3}
\end{equation}
but, on the other hand,
\begin{equation}
\e\big\la (R_{1,2}^\infty)^2 \big\ra = \big(\e\la R_{1,2}^\infty  \ra\big)^2\quad \text{and}\quad
\e \e_\diamond  \frac{\rrac[a]{f_n \dT_1^1 e^{\sum_{\ell\le n} \theta_{1}^\ell}}}{\rrac[a]{e^{\theta_1}}^n}
= \e \e_\diamond \rrac[a]{f_n}
\e \e_\diamond\frac{\rrac[a]{\dT_1 e^{\theta_1}}}{\rrac[a]{e^{\theta_1}}}
\label{eqFdSpertf4}
\end{equation}
for any function $f_n$ of finitely many spins $\sigma_i^\ell, \sigma_i^\diamond$ of finitely many replicas $\ell$ and the special replica $\diamond$ with $2\leq i\leq m$ for some $m$, $\xi\sim \mathrm{Exp}(1)$, and
$$
\theta_1^\ell=\ln(1+\lambda \sigma_1^\ell)-\lambda y_1 \sigma_1^\ell ,\quad
y_1=\frac{\xi}{1+\lambda \sigma_1^\diamond},\quad
\dT_1^\ell= \frac{y_1\sigma_1^\ell}{1+\lambda \sigma_1^\diamond},
$$
with $\lambda\in \Lambda:=\{\lambda_k: k\geq 1\}$. The reason we exclude spin index $i=1$ in the coordinates of $f_n $ is to reserve it specifically for $(\sigma_1^\ell)_{\ell \le n}$, because before the limit the spins $(\sigma_i^\ell)_{\ell\le n}$ appearing explicitly in Theorem~\ref{ThFdSiden} depended on a uniform random index $i\in\{1,\ldots,N\}$, which by symmetry can be fixed to $1$ as long as we avoid the spin indices on which $f_n $ depends (because a random uniformly chosen index $i\leq N$ belongs to $\{2,\ldots, m\}$ with vanishing probability in the limit).

\subsection{Thermal pure state} \label{SecAHb2} The identity $\e\la (R_{1,2}^\infty)^2 \ra = (\e\la R_{1,2}^\infty \ra)^2$ means that
$$
R_{1,2}^\infty = R_{1,2}^\infty(w,u_1,u_2) = \int_0^1 \bs(w,u_1,v)\bs(w,u_2,v)\,dv
$$
is constant and is thus almost surely independent of $w,u_1,u_2$. This means that, in fact, the function $\bs(w,u,v) = \bs(w,v)$ almost surely (i.e., it does not depend on $u$) and  $\int_0^1\! \bs(w,v)^2\,dv = \mathrm{const}$: the system is said to lie in a ``thermal pure state''. This appears, for example, in Theorem~5 in \cite{RSKsat} or Lemma~1 in \cite{1RSB}, and can be explained in a few words. Indeed, if we consider a (random) measure $du\circ(u\mapsto \bs(w,u,\,\cdot\,))^{-1}$ on $(L^2[0,1],dv)$, the concentration of the overlap means that the scalar product between two points (functions in $L^2$) sampled from this measure is constant, which means that the measure concentrates on one (random) function $\bs(w,\,\cdot\,)$ on the sphere of some fixed constant radius in $L^2$. In particular, instead of \eqref{AHrepr} we now have
\begin{equation}
(\sigma_i^\ell)_{i,\ell\ge 1} \stackrel{\rm d}{=}\big(\sigma(w,v_i,x_{i,\ell})\big)_{i,\ell\ge 1}
\label{Ocon}
\end{equation}
for some (any) function $\sigma$ of three variables such that $\int_0^1\! \sigma(w,v,x)\, dx = \bs(w,v)$; see also the proof of Corollary \ref{cor:asymp_spin_dist} in Section~\ref{sec:proof_coro} for a similar argument. At the level of the asymptotic spin array $(\sigma_i^\ell)_{i,\ell\ge 1}$ this means that the expectation $\la \, \cdot\, \ra$ is now asymptotically equivalent to a simple integral over $(x_{i,\ell})$ only, and that the replica indices can be freely exchanged. Writting a (finite) joint moment of spins gives, asymptotically,
\begin{align*}
 \e\big\langle \prod_{(i,\ell)\in \mathcal{C}} \sigma_i^\ell\big\rangle = \e\prod_{(i,\ell)\in \mathcal{C}} \bs(w,u_\ell,v_i)=\e\prod_{(i,\ell)\in \mathcal{C}} \bs(w,v_i)
\end{align*}
(here we asssumed for simplicity that $\mathcal{C}$ does not contain repeated elements) which, e.g., concretely implies that $\e \langle\sigma_1^1\sigma_1^2\sigma_2^1\sigma_2^2\rangle=\e\langle \sigma_1\sigma_2\rangle^2$ is asymptotically also equal to $\e \langle\sigma_1^1\sigma_1^2\sigma_2^3\sigma_2^4\rangle=\e(\langle \sigma_1\rangle \langle \sigma_2\rangle)^2$. Also, the asymptotic multioverlaps therefore simplify to
\begin{equation}
R^\infty_{\ell_1,\ldots, \ell_n}(w)= R^\infty_{1,\ldots, n}(w) =\int_0^1 \bs(w,v)^n\,dv\,. \label{asymptMultiover_simple}
\end{equation}

\subsection{Concentration of multioverlaps, $n\geq 3$.}  \label{SecAHb3}
To see how concentration of the overlap in the form \eqref{Ocon} implies concentration of all multioverlaps, let us first derive the following consequence of the identities \eqref{eqFdSpertf4}.

\begin{lemma}[A decoupling lemma]
If $\xi_1,\xi_2\sim \operatorname{Exp}(1)$ are independent and, for $j=1,2,$  
$$
\theta_j:=\ln(1+\lambda \sigma_j)-\lambda y_j \sigma_j\,,\quad
y_j:=\frac{\xi_j}{1+\lambda \sigma_j^\diamond}\,,\quad
\dT_j:= \frac{y_j\sigma_j}{1+\lambda \sigma_j^\diamond}\,,
$$
where $\lambda\in \Lambda=\{\lambda_k : k\geq 1\}$ then
\begin{equation}
\e\e_\diamond \frac{\rrac[a]{\dT_1 e^{\theta_1} \dT_2 e^{\theta_2}}}{\rrac[a]{e^{\theta_1} e^{\theta_2}}}
= 
\e\e_\diamond \frac{\rrac[a]{\dT_1 e^{ \theta_1}}}{\rrac[a]{e^{ \theta_1}}}
\e\e_\diamond \frac{\rrac[a]{\dT_2 e^{\theta_2}}}{\rrac[a]{e^{ \theta_2}}}\,.
\label{FdSlim1r3}
\end{equation}
\end{lemma}
\begin{proof}
Let us take large $M\gg 1$ and consider a set $A=\{\xi_2:0\le \xi_2\leq M\}$. On this set, we have the bound $|\theta_2|\leq M_\lambda$ for some constant $M_\lambda$. If we write
$$
\frac{\rrac[a]{\dT_1 e^{  \theta_1} \dT_2 e^{ \theta_2}}}{\rrac[a]{e^{  \theta_1} e^{  \theta_2}}}
=
\frac{\rrac[a]{\dT_1 e^{  \theta_1} \dT_2 e^{ \theta_2}}/\la e^{  \theta_1}\ra}{\rrac[a]{e^{  \theta_1} e^{  \theta_2}}/\la e^{  \theta_1}\ra}\,,
$$
note that on the set $A$ the denominator is in the interval $[e^{-3M_\lambda},e^{3M_\lambda}]$, and approximating $1/x$ on this interval by a polynomial $\sum_{n=0}^r c_n x^n$ uniformly within error $\eps$, we get
\begin{align*}
\e\e_\diamond \frac{\rrac[a]{\dT_1 e^{  \theta_1} \dT_2 e^{ \theta_2}}}{\rrac[a]{e^{  \theta_1} e^{  \theta_2}}} \boldsymbol{1}(\xi_2\in A)
&\approx
\sum_{n=0}^r c_n \e\e_\diamond
\frac{\rrac[a]{\dT_1 e^{\theta_1} \dT_2 e^{\theta_2}}}{\la e^{  \theta_1}\ra}
 \Brac[r]{\frac{\rrac[a]{e^{  \theta_1} e^{  \theta_2}}}{\la e^{  \theta_1}\ra}}^n \boldsymbol{1}(\xi_2\in A)\,.
\end{align*}
We can represent the $n$th term on the right-hand side using replicas as
$$
c_n\e\e_\diamond
\frac{\rrac[a]{\dT_1^1 e^{ \sum_{\ell \le n+1}\theta_1^\ell} \dT_2^1 e^{ \sum_{\ell \le n+1}\theta_2^\ell}}}{\la e^{\theta_1}\ra^{n+1}}\boldsymbol{1}(\xi_2\in A)\,.
$$
We can then apply \eqref{eqFdSpertf4} with the function $f_{n+1}:=d_2^1\exp\sum_{\ell\le n+1}\theta_2^\ell$ for a fixed $\xi_2$ first and average over $\xi_2$ to rewrite this as
$$
c_n\e\e_\diamond \frac{\rrac[a]{\dT_1 e^{  \theta_1}}}{\rrac[a]{e^{  \theta_1}}}
\e\e_\diamond \rrac[a]{\dT_2 e^{ \theta_2}}
\rrac[a]{e^{  \theta_2}}^n \boldsymbol{1}(\xi_2\in A)\,.
$$
Summing all the terms and again using that $|1/x- \sum_{n=0}^r c_n x^n|\leq \eps$ on the interval $[e^{-3M_\lambda},e^{3M_\lambda}]$, we showed that (within error $2\eps$)
$$
\e\e_\diamond \frac{\rrac[a]{\dT_1 e^{  \theta_1} \dT_2 e^{ \theta_2}}}{\rrac[a]{e^{  \theta_1} e^{  \theta_2}}} \boldsymbol{1}(\xi_2\in A)
\approx
\e\e_\diamond \frac{\rrac[a]{\dT_1 e^{  \theta_1}}}{\rrac[a]{e^{  \theta_1}}}
\e\e_\diamond \frac{\rrac[a]{\dT_2 e^{  \theta_2}}}{\rrac[a]{e^{  \theta_2}}}\boldsymbol{1}(\xi_2\in A)\,.
$$
Letting $\eps\downarrow 0$ and then letting $M\uparrow\infty$ finishes the proof.
\end{proof}

We finally have all ingredients to finish the proof of Theorem \ref{MainTh}.

\begin{proof}[Proof of Theorem \ref{MainTh}]
The rest of the proof of Theorem \ref{MainTh} is very similar in spirit to the calculations in \cite{1RSB, finiteRSB}.  

Recall that $\sigma_j = \sigma(w,v_j,x_j)$ and $\sigma_j^\diamond = \sigma(w,v_j,x_j^\diamond)$ and we can interpret $\la\,\cdot\,\ra$ as the expectation with respect to $(x_j)$, $\e_\diamond$ as the expectation in $(x_j^\diamond)$ and $\e$ as the expectation in $w$, $(\xi_j)$ and $(v_j)$. Since all random variables indexed by $j=1,2$ are independent, if we denote by $\e_{|w}$ the conditional expectation given $w$ (that includes $\e_{\diamond}$) then the left-hand side of \eqref{FdSlim1r3} can be written as
$$
\e\Bigl( \e_{|w} \frac{\rrac[a]{\dT_1 e^{\theta_1}}}{\rrac[a]{e^{\theta_1}}}\Bigr)
\Bigl(\e_{|w}\frac{\rrac[a]{ \dT_2 e^{\theta_2}}}{\rrac[a]{e^{\theta_2}}}\Bigr)
=
\e\Bigl( \e_{|w} \frac{\rrac[a]{\dT_1 e^{\theta_1}}}{\rrac[a]{e^{\theta_1}}}\Bigr)^2.
$$
If we let 
$$
Y=Y(w):=  \e_{|w} \frac{\rrac[a]{\dT_1 e^{\theta_1}}}{\rrac[a]{e^{\theta_1}}}
= \e_{|w} \frac{y_1}{1+\lambda \sigma_1^\diamond} \frac{\rrac[a]{\sigma_1 e^{\theta_1}}}{\rrac[a]{e^{\theta_1}}}
$$ 
then \eqref{FdSlim1r3} gives that $\Var(Y(w))=0$ and $Y=\e Y$ almost surely. Let us write down $Y$ a bit more explicitly. Taking the expectation in $\xi_1$ first (which, again, does not appear in $\langle\, \cdot\,\rangle$) and using that, conditionally on $\sigma_1^\diamond$,
$$
y_1 = \frac{\xi_1}{1+\lambda \sigma_1^\diamond}\sim \operatorname{Exp}(1+\lambda\sigma_1^\diamond)\,,
$$
we can write
\begin{align*}
Y(w)&=\e_{|w}\int_0^\infty (1+\lambda\sigma_1^{\diamond})e^{-(1+\lambda\sigma_1^\diamond)y} \frac{y}{1+\lambda\sigma_1^{\diamond}}
\frac{\rrac[a]{\sigma_1 e^{\theta_1}}}{\rrac[a]{e^{\theta_1}}}\, dy
\\
&=\e_{|w}\int_0^\infty e^{-\lambda\sigma_1^\diamond y} \frac{\rrac[a]{\sigma_1 e^{\theta_1}}}{\rrac[a]{e^{\theta_1}}}ye^{-y}\, dy
\\
&=\e_{|w}\int_0^\infty e^{-\lambda\sigma_1^\diamond y} \frac{\rrac[a]{\sigma_1 (1+\lambda\sigma_1)e^{-\lambda y\sigma_1}}}{\rrac[a]{(1+\lambda\sigma_1)e^{-\lambda y\sigma_1}}}ye^{-y}\, dy
\\
&=\e_{|w}\int_0^\infty \rrac[a]{e^{-\lambda\sigma_1 y}} \frac{\rrac[a]{\sigma_1 (1+\lambda\sigma_1)e^{-\lambda y\sigma_1}}}{\rrac[a]{(1+\lambda\sigma_1)e^{-\lambda y\sigma_1}}}ye^{-y}\, dy\,,
\end{align*}
where in the last line we used that the expectation $\e_\diamond$ of $\sigma_1^\diamond = \sigma(w,v_1,x_1^\diamond)$ with respect to $x_1^\diamond$ is the same as the expectation $\la\,\cdot\,\ra$ of $\sigma_1 = \sigma(w,v_1,x_1)$ with respect to $x_1$. To simplify the notation, let us now omit index $1$, write $\sigma = \sigma(w,v,x)$, 
$$
Y(w)= \e_{|w}\int_0^\infty \rrac[a]{e^{-\lambda\sigma y}} \frac{\rrac[a]{\sigma (1+\lambda\sigma)e^{-\lambda y\sigma}}}{\rrac[a]{(1+\lambda\sigma)e^{-\lambda y\sigma}}}ye^{-y}\, dy\,,
$$
where $\la\,\cdot\,\ra$ is the expectation with respect to $x$.

Up to this point we only considered $\lambda\in \Lambda$ and showed that $Y(w)=\e Y$ almost surely for all such $\lambda.$ Therefore, for $w$ in a set $\mathcal{A}$ of probability one,  $Y(w)=\e Y$ for all $\lambda\in \Lambda.$ However, both 
$$ g_w:\gamma\mapsto g_w(\gamma):=\e_{|w}\int_0^\infty \rrac[a]{e^{-\gamma\sigma y}} \frac{\rrac[a]{\sigma (1+\gamma\sigma)e^{-\gamma y\sigma}}}{\rrac[a]{(1+\gamma\sigma)e^{-\gamma y\sigma}}}ye^{-y}\, dy$$
for a fixed $w$, as well as its $w$-expectation $\e g_w(\gamma)$ are analytic functions of $\gamma$ in a small neighbourhood of $0$ (the measure defining the bracket $\langle \, \cdot\, \rangle$ appearing in the definition of $g_w(\gamma)$ does not depend on $\gamma$, it depends instead on the set $\Lambda$ of limiting values of the sequence $(\lambda_k^N)$). Since $g_w(\gamma)-\e g_w(\gamma)=0$ almost surely for all $\gamma \in \Lambda$ where the set $\Lambda$ accumulates at $0$, then $g_w(\gamma)=\e g_w(\gamma)$ almost surely for all $\gamma$ in a small neighbourhood of zero, namely, for all $w\in\mathcal{A}$, the equation $Y(w)=\e Y$ holds for all $\lambda$ in a small neighbourhood of zero. From now on we assume that $w\in\mathcal{A}$.

Since 
$$
Z(w):= \e_{|w}\int_0^\infty \rrac[a]{\sigma (1+\lambda\sigma)e^{-\lambda y\sigma}} ye^{-y}\, dy
$$
is a linear function of the magnetisation $\e_v \bs(w,v)$ (by Taylor expansion and using that $\sigma\in\{-1,1\}$ so that no power of $\sigma$ other than $1$ appear) and, thus, independent of $w$ (we already proved in Section~\ref{sec:magn} that the magnetisation concentrates),
$$
X(w):=\frac{Z(w)-Y(w)}{\lambda}= \e_{|w}\int_0^\infty \rrac[a]{\sigma e^{-\lambda y\sigma }} \frac{\rrac[a]{\sigma (1+\lambda\sigma)e^{-\lambda y\sigma}}}{\rrac[a]{(1+\lambda\sigma)e^{-\lambda y\sigma}}}ye^{-y}\, dy
$$
is almost surely constant as well for all $\lambda$ in a small neighbourhood of zero. 

At $\lambda=0$ this equals to the overlap $\e_{|w}\la\sigma\ra^2 = R_{1,2}^\infty(w)$, which we already knew to concentrate. We will now show that the fact that all derivatives in $\lambda$ at zero of $X(w)$ are independent of $w$ implies that all multioverlaps concentrate. Given $n\geq 1,$ when we compute $\frac{\partial^n}{\partial \lambda^n}$, every time we apply derivative to the denominator we use that 
$$
\frac{\partial}{\partial\lambda}\frac{1}{\rrac[a]{(1+\lambda\sigma)e^{-\lambda y\sigma}}^k} = - k\frac{\rrac[a]{[\sigma - y\sigma (1+\lambda\sigma)]e^{-\lambda y\sigma}}}{\rrac[a]{(1+\lambda\sigma)e^{-\lambda y\sigma}}^{k+1}}\,.
$$
Notice that the numerator at $\lambda=0$ of this last expression is equal to $(1-y)\la\sigma\ra$, so applying derivative to the denominator in $X$'s expression brings out another ``replica'' $\la\sigma\ra$. In fact, applying $\frac{\partial^n}{\partial \lambda^n}$ to $X$'s denominator will produce the term (inside the integral)
$$
(-1)^n n!
\rrac[a]{\sigma e^{-\lambda y\sigma }} \frac{\rrac[a]{\sigma (1+\lambda\sigma)e^{-\lambda y\sigma}}}{\rrac[a]{(1+\lambda\sigma)e^{-\lambda y\sigma}}^{n+1}}
\rrac[a]{[\sigma - y\sigma (1+\lambda\sigma)]e^{-\lambda y\sigma}}^n\,,
$$
which at $\lambda=0$ equals $(-1)^n n! (1-y)^n \la\sigma\ra^{n+2}.$ After integration, this term equals (with $\xi\sim \operatorname{Exp}(1)$)
$$
n!\, \e_{|w}\la\sigma\ra^{n+2}\int_0^\infty  (y-1)^n ye^{-y}\, dy
 = n! \,R_{1,\ldots,n+2}^\infty\, \e \xi(\xi-1)^n\,.
$$
Note that the coefficient $\e \xi(\xi-1)^n = \e (\xi-1)^{n+1} + \e (\xi-1)^n > 0$ for $n\geq 1$. What we have just discussed is the term obtained by applying all derivatives to the denominator only which is inside the integral in $X(w)$'s expression. If along the way we apply a derivative in $\lambda$ to any factor instead in the numerator, this will not create a new replica, so all those terms will produce a linear combination of multioverlaps on strictly less than $n+2$ replicas, which by induction we assume to be independent of $w$. This implies that $R_{1,\ldots,n+2}^\infty$ is independent of $w$.

Therefore, all multioverlaps are constant in this particular subsequential thermodynamic limit, $R_{1,\ldots,n}^\infty(w)=\e R_{1,\ldots,n}^\infty$ almost surely. This contradicts \eqref{MOconcentrationPf3} which was a consequence of assuming the existence of a subsequence along which \eqref{MOconcentrationPf} holds, i.e., along which $\e_{\lambda}\e\la (R_{1,\ldots,n}  - \e\la R_{1,\ldots,n} \ra)^2\ra \geq \delta$ for any $\delta >0$. Therefore such subsequence does not exist, which proves Theorem~\ref{MainTh}.
\end{proof}

Concentration of multioverlaps means that the distribution $dv\circ(v\mapsto  \bs(w,v))^{-1}$ is almost surely independent of $w$ and is equal to some $\zeta\in\Pr[-1,1]$. As a result, as we mentioned below Theorem~\ref{MainTh}, the spins $\sigma_i^\ell$ (in this subsequential limit) can be generated by taking an i.i.d. sequence $m_i\sim\zeta$ and then flipping independent $\pm 1$ valued coins with expected value $m_i$ to output $(\sigma_i^\ell)_{\ell \ge 1}$.

\section{The case of soft spins: proof of Theorem~\ref{MainThG}}\label{ProofSoft}
The proof for soft bounded spins follows closely the one for Ising spins, so we will be more brief.
\subsection{Magnetisation, $n=1$, and generalised overlaps, $n=2$} Because of the assumption of factorised prior \eqref{ProductM}, the proof of concentration of the magnetisation $R_{1}$ is identical to the one provided in Section~\ref{sec:magn}. 

The proof of Theorem~\ref{GenOVer} is a straightforward extension of the one of Theorem~\ref{OverlapConc}. Indeed, simply notice that the key relation \eqref{remarkable} can be easily extended to the generalised overlap \eqref{gene_overlap} based on $\mathcal{L}^{(k)}:=N^{-1}d\mathcal{H}^{\rm gauss}/d\lambda_{0k}$ with Hamiltonian \eqref{OverPertG}. Then, given index $k$ in the perturbation Hamiltonian \eqref{OverPertG}, after redefining $(\sigma^*_i)^k,(\sigma^\ell_i)^k$ as generalised spins (still certainly taking values in $[-1,1]$), the proof of concentration of $\mathcal{L}^{(k)}$ is identical to the one of Proposition~\ref{L-concentration}, when using the change of variable $[2^{-k-1},2^{-k}]\ni\lambda_{0k}\to \lambda_{0k}2^{-k}$ with $\lambda_{0k}\in[1/2,1]$ and then absorbing the $2^{-k}$ in $\eps_N$.

\subsection{Multioverlaps, $n\ge 3$}
The key result is again the following modification of Theorem~\ref{ThFdSiden}, whose proof is identical and will not be repeated.
\begin{theorem}[Franz-de Sanctis identities in inference, soft spins case]\label{ThFdSidenG}
Let $i$ be a random uniform index in $\{1,\ldots,N\}$ and, for any $I\in \II$ and $\xi\sim \mathrm{Exp}(1)$, 
$$
\theta_{iI}^\ell:=\ln(1+ \lambda_I P_I(\sigma_{i}^\ell))- \lambda_Iy_{iI}  P_I(\sigma_{i}^\ell) \,,\quad
y_{iI}=\frac{\xi}{1+\lambda_I P_I(\sigma_{i}^*)}\,,\quad
\dT_{iI}^\ell= \frac{y_{iI}  P_I(\sigma_i^\ell)}{1+\lambda_I P_I(\sigma_{i}^*)}\,.
$$
Under the same hypotheses as in Theorem~\ref{ThFdSiden} we have
\begin{equation}
\e_{\lambda} \left|
\e \frac{\la{f_n  \dT_{iI}^1 e^{\sum_{\ell\le n} \theta_{iI}^\ell}}\ra}{\la{e^{  \theta_{iI}}}\ra^n}
- \e \la{f_n }\ra
\e \frac{\la{\dT_{iI}e^{  \theta_{iI}}}\ra}{\la{e^{ \theta_{iI}}}\ra}
\right| \leq 
C_I \Big\{\frac{1}{s_N}+\Big(\frac{v_N N}{s_N^2}\Big)^{1/3}\Big\}^{1/2}\,.
\label{eqFdSpertfThG}
\end{equation}
\end{theorem}

\subsection{Passing to the limit} Suppose that \eqref{MOconcentrationG} fails for some indices $k_1,\ldots,k_n$, namely, that there exists some subsequence $(N_j)_{j\geq 1}$ along which
\begin{equation}
\e_{\lambda}
\e\big\la (R_{1,\ldots, n}^{(k_1,\ldots,k_n)} - \e\la R_{1,\ldots, n}^{(k_1,\ldots,k_n)} \ra)^2\big\ra \geq \delta>0
\label{MOconcentrationPfA}
\end{equation}
for some (any) $\delta>0$. On the other hand, \eqref{eqFdSpertfThG} holds for any set of arguments of $f_n$ and $I\in \II$. Since this is a countable collection, the equations \eqref{Oconcentration2G}, \eqref{eqFdSpertfThG} and \eqref{MOconcentrationPfA} imply that we can choose some $\lambda=\lambda^N$ varying with $N$, with $\lambda_{0k}^N\in [2^{-k-1},2^{-k}]$ and $\lambda_I^N\in [1/2,1]$, such that, along the same subsequence $(N_j)_{j\geq 1}$,
\begin{equation}
\e\big\la (R_{1,\ldots, n}^{(k_1,\ldots,k_n)} - \e\la R_{1,\ldots, n}^{(k_1,\ldots,k_n)} \ra)^2\big\ra \geq \frac{\delta}{2}>0\,,
\label{MOconcentrationPf2G}
\end{equation}
and
\begin{equation}
\e\big\la (R_{1,2}^{(k)}- \e\la R_{1,2}^{(k)} \ra)^2\big\ra \to 0\,,\qquad
\left|
\e \frac{\la{f_n  \dT_{iI}^1 e^{\sum_{\ell\le n} \theta_{iI}^\ell}}\ra}{\la{e^{  \theta_{iI}}}\ra^n}
- \e \la{f_n }\ra
\e \frac{\la{\dT_{iI}e^{\theta_{iI}}}\ra}{\la{e^{ \theta_{iI}}}\ra}
\right|\to 0
\label{eqFdSpertf32G}
\end{equation} 
jointly for all possible arguments of $f_n$ and $I\in\II$, where now the Gibbs measure $\la\,\cdot\,\ra$ is for the Hamiltonian with these specific parameters $\lambda^N$ (see Section~\ref{sec:passing_limit} for a detailed proof of existence of such $\lambda^N$). 

As in the binary case, by the Nishimori identity, we can replace the disorder $\sigma^*$ in all the integrands in the second equation in \eqref{eqFdSpertf32G} by another replica. For convenience of notation, we will denote this new replica by $\sigma^\diamond$ to distinguish from the disorder $\sigma^*$ and at the same time not to occupy any specific index. Then, the second equation in \eqref{eqFdSpertf32G} can be written as
\begin{equation}
\left|
\e\e_\diamond \frac{\la{f_n  \dT_{iI}^1 e^{\sum_{\ell\le n} \theta_{iI}^\ell}}\ra}{\la{e^{  \theta_{iI}}}\ra^n}
- \e\e_\diamond \la{f_n }\ra
\e\e_\diamond \frac{\la{\dT_{iI}^1e^{\theta_{iI}}}\ra}{\la{e^{ \theta_{iI}}}\ra}
\right|\to 0
\label{eqFdSpertf32Gcp}
\end{equation}
where $\e_\diamond$ denotes the Gibbs average $\la\,\cdot\,\ra$ with respect to the replica $\sigma^\diamond$ only, $\la\,\cdot\,\ra$ denotes the Gibbs average with respect to all other replicas, $f_n$ is a function of finitely many spins on $n$ replicas and on $\sigma^\diamond$, and
$$
\theta_{iI}^\ell=\ln(1+ \lambda_I^N P_I(\sigma_{i}^\ell))- \lambda_I^Ny_{iI}  P_I(\sigma_{i}^\ell) \,,\quad
y_{iI}=\frac{\xi}{1+\lambda_I^N P_I(\sigma_{i}^\diamond)}\,,\quad
\dT_{iI}^\ell= \frac{y_{iI}  P_I(\sigma_i^\ell)}{1+\lambda_I^N P_I(\sigma_{i}^\diamond)}\,.
$$
Now, by Cantor's diagonalisation let us extract a further subsequence of $(N_j)_{j\geq 1}$ such that all $\lambda^{N}_{0k} \to \lambda_{0k}\in  [2^{-k-1},2^{-k}]$ and $\lambda_I^N\to \lambda_I\in [1/2,1]$ converge and, moreover, the distribution of all spins on all replicas $(\sigma_i^\ell)_{i,\ell\geq 1}$ under $\e\la\,\cdot\,\ra$ also converges weakly in the finite-dimensional sense along this subsequence.

\subsection{Aldous-Hoover representation.}\label{SecAHsoft1}
In the case of soft spins, the Aldous-Hoover representation \eqref{AHrepr} can be  expressed in terms of multioverlaps as follows. If we denote
$$
\bs^{(k)}(w,u,v):=\int_0^1\! \sigma(w,u,v,x)^k\, dx
$$
then the asymptotic analogue of the multioverlap above is
\begin{equation}
R_{\ell_1,\ldots, \ell_n}^{(k_1,\ldots,k_n)}\to R_{\ell_1,\ldots, \ell_n}^{(k_1,\ldots,k_n)\infty}(w,(u_{\ell_j})_{j\le n}) := \int_0^1 \prod_{j\leq n} \bs^{(k_j)}(w,u_{\ell_j},v)\,dv\,,
\end{equation}
in the sense that joint moments of all multioverlaps before the limit converge to joint moments of these analogues in the limit. For the limiting generalised overlap 
\begin{align}
R_{1,2}^{(k)}\to R_{1,2}^{(k)\infty}(w,u_1,u_2) := \int_0^1 \bs^{(k)}(w,u_{1},v)\bs^{(k)}(w,u_{2},v)\,dv\,.  \label{def:geneOver}
\end{align}
As before, \eqref{MOconcentrationPf2G} and \eqref{eqFdSpertf32G} become in this subsequential limit  
\begin{equation}
\e\big\la (R_{1,\ldots, n}^{(k_1,\ldots,k_n)\infty} - \e\la R_{1,\ldots, n}^{(k_1,\ldots,k_n)\infty} \ra)^2\big\ra \geq \frac{\delta}{2}>0
\label{MOconcentrationPf2Ga}
\end{equation}
together with
\begin{equation}
\e \big\la (R_{1,2}^{(k)\infty})^2 \big\ra = \big(\e\la R_{1,2}^{(k)\infty}  \ra\big)^2\,,\qquad
\e\e_\diamond  \frac{\la{f_n  \dT_{1I}^1 e^{\sum_{\ell\le n} \theta_{1I}^\ell}}\ra}{\la{e^{  \theta_{1I}}}\ra^n}
= \e_\diamond \e \la{f_n }\ra\,
\e_\diamond \e \frac{\la{\dT_{1I}e^{ \theta_{1I}}}\ra}{\la{e^{ \theta_{1I}}}\ra}
\label{eqFdSpertf32Ga}
\end{equation} 
for any bounded function $f_n$ of finitely many spins $\sigma_i^\diamond, (\sigma_i^\ell)$ for $2\leq i\leq m$ and $\ell\le n$ for some $m$ and $n$, and
$$
\theta_{1I}^\ell=\ln(1+ \lambda_I P_I(\sigma_{1}^\ell))- \lambda_I y_{1I}  P_I(\sigma_{1}^\ell) \,,\quad
y_{1I}=\frac{\xi}{1+\lambda_I P_I(\sigma_{1}^\diamond)}\,,\quad
\dT_{1I}^\ell= \frac{y_{1I} P_I(\sigma_1^\ell)}{1+\lambda_I P_I(\sigma_{1}^\diamond)}\,,
$$
for some $\lambda_I\in [1/2,1]$, $\xi\sim {\rm Exp}(1)$ independently of everything else.

\subsection{Thermal pure state} \label{SecAHsoft2}
As in the Ising case, the fact that the limiting generalised overlaps \eqref{def:geneOver} concentrate (by Theorem~\ref{GenOVer}), namely $R_{1,2}^{(k)\infty}(w,u_1,u_2)=\e R_{1,2}^{(k)\infty}$ almost surely, means that $\bs^{(k)}(w,u,v) = \bs^{(k)}(w,v)$ almost surely (i.e., the function does not depend on $u$) and 
$$
\int_0^1 \bs^{(k)}(w,v)^2\,dv = c_k\,.
$$
Indeed, if we consider a (random) measure $du\circ(u\mapsto \bs^{(k)}(w,u,\,\cdot\,))^{-1}$ on $(L^2[0,1],dv)$, the concentration of the overlap means that the scalar product between two points (functions in $L^2$) sampled from this measure is constant, which means that the measure concentrates on one (random) function $\bs^{(k)}(w,\,\cdot\,)$ on the sphere of fixed constant radius $\sqrt{c_k}$ in $L^2$. This implies that
\begin{equation}
(\sigma_i^\ell)_{i,\ell\geq 1} \stackrel{\rm d}{=}\big(\sigma(w,v_i,x_{i,\ell})\big)_{i,\ell\geq 1}
\label{AHso}
\end{equation}
for some (any) function $\sigma$ of three variables such that
$$
\bs^{(k)}(w,u,v)=\bs^{(k)}(w,v)=\int_0^1 \sigma(w,v,x)^k\, dx\,.
$$
In particular the multioverlaps verify
$$R_{\ell_1,\ldots, \ell_n}^{(k_1,\ldots,k_n)\infty}= R_{1,\ldots, n}^{(k_1,\ldots,k_n)\infty}(w)\,.$$
\subsection{Concentration of multioverlaps, $n\geq 3$.}\label{SecAHsoft3}
From \eqref{eqFdSpertf32Ga} we derive the analogue of \eqref{FdSlim1r3} by an identical proof.
\begin{lemma}[A decoupling lemma]
If $\xi_1,\xi_2\sim \operatorname{Exp}(1)$ are independent then, for all $I\in\II$,
\begin{equation}
\e\e_\diamond \frac{\rrac[a]{\dT_1 e^{\theta_1} \dT_2 e^{\theta_2}}}{\rrac[a]{e^{\theta_1} e^{\theta_2}}}
= 
\e\e_\diamond \frac{\rrac[a]{\dT_1 e^{ \theta_1}}}{\rrac[a]{e^{ \theta_1}}}
\e\e_\diamond \frac{\rrac[a]{\dT_2 e^{\theta_2}}}{\rrac[a]{e^{ \theta_2}}}\,,
\label{FdSlim1r3S}
\end{equation}
where, for some $\lambda\in [1/2,1]$, 
$$
\theta_j:=\ln(1+ \lambda P_I(\sigma_j))- \lambda y_j  P_I(\sigma_j) \,,\quad
y_j:=\frac{\xi_j}{1+\lambda P_I( \sigma_j^\diamond)}\,,\quad
\dT_j:= \frac{y_j P_I(\sigma_j)}{1+\lambda P_I( \sigma_j^\diamond)}\,.
$$
\end{lemma}

\begin{proof}[Proof of Theorem \ref{MainThG}] To finish the proof, we will arrive at a contradiction with \eqref{MOconcentrationPf2Ga}. The proof is identical to the one in the binary case up to the point were (here again $\sigma = \sigma(w,v,x)$ and $\langle\, \cdot\,\rangle$ is the $x$-expectation)
$$
X(w):= \e_{|w}\int_0^\infty \rrac[a]{P(\sigma) e^{-\lambda P(\sigma)y}} \frac{\rrac[a]{P(\sigma) (1+\lambda P(\sigma))e^{- \lambda P(\sigma)y}}}{\rrac[a]{(1+\lambda P(\sigma))e^{-\lambda P(\sigma)y}}}ye^{-y}\, dy
$$
is almost surely constant, for all polynomials with coefficients bounded by $1$ and small enough $\lambda$. By taking derivatives in $\lambda$ as in the binary case, we get by induction that 
$$
\e_{|w}\la P(\sigma)\ra^{n} = \e_v \la P(\sigma)\ra^{n}
$$ 
are independent of $w$. By expanding this in the coefficients of $P(\sigma)$, we get that all multioverlaps $\e_v \prod_{\ell\le n} \la\sigma^{k_\ell}\ra = R_{1,\ldots,n}^{(k_1,\ldots, k_n)\infty}$ concentrate. This contradicts \eqref{MOconcentrationPf2Ga}, which finishes the proof.
\end{proof}

In this case, the concentration of multioverlaps means that we can redefine $\sigma(w,v,x)$ in \eqref{AHso} and find a function $\sigma(v,x)$ such that the array $(\sigma_i^\ell)_{i,\ell\geq 1}$ is equal in distribution to $(\sigma(v_i,x_{i,\ell}))_{i,\ell\geq 1}$, see similar arguments in Section~\ref{sec:proof_coro} below. The distribution of this array can be encoded via a random measure $\mu(v):=dx\circ(x\mapsto  \sigma(v,x))^{-1}\in \Pr[-1,1]$ with the distribution $\zeta\in \Pr(\Pr[-1,1])$.

\section{Proof of the corollaries: asymptotic representation and decoupling}\label{sec:proof_coro}

As a direct consequence of Theorems \ref{MainTh} and \ref{MainThG}, the proofs of the corollaries follow. We focus on the case where the spins are in $[-1,1]$, but nothing changes when they are binary.

\begin{proof}[Proof of Corollary \ref{cor:asymp_spin_dist}]
    Let $\lambda^N$ be an appropriate sequence of perturbation parameters allowing for multioverlap concentration (which exist by Theorem \ref{MainThG}). Take a subsequence along which it converges, with limit $\lambda$. Multioverlap concentration means that along every subsubsequence of system sizes s.t. also the spin variables $(\sigma_i^\ell)_{i,\ell\geq1}$ converge in distribution, the associated asymptotic multioverlaps $R_{\ell_1,\ldots, \ell_n}^{(k_1,\ldots,k_n)\infty}(w,(u_{\ell_j})_{j\le n})$ are a.s. constant for all $(k_1,\ldots,k_n)$. This implies that there exist fixed values $w_0\in[0,1]$ and $(u_{0,\ell})_{\ell\geq1}\in[0,1]^\N$ s.t. for all $(k_1,\dots,k_n)$, $$R_{\ell_1,\ldots, \ell_n}^{(k_1,\ldots,k_n)\infty}(w,(u_{\ell_j})_{j\le n}) = R_{\ell_1,\ldots, \ell_n}^{(k_1,\ldots,k_n)\infty}(w_0,(u_{0,\ell_j})_{j\le n})$$ almost surely. 

    We now prove that we may take the Aldous-Hoover limit to be given (for the spin $\sigma_i^\ell$, with $i,\ell\geq1$) by a new function $$\tilde \sigma_\lambda(v_i,x_{i,\ell}) := \sigma(w_0,u_{0,\ell},v_i,x_{i,\ell}).$$ By doing so, the values of the means of all the multioverlaps and their products remain unchanged. This implies that the joint distribution of the variables $(\tilde \sigma_\lambda(v_i,x_{i,\ell}))_{i,\ell\ge 1}$ is the same as the one of $(\sigma(w,u_\ell,v_{i},x_{i,\ell}))_{i,\ell\ge 1}$. Indeed, recall that any generic joint moment of spins can be straightforwardly expressed as the mean of a product of multioverlaps. By what we said previously these moments are asymptotically the same when computed using representation $\sigma$ (which may depend on the considered subsubsequence, and thus on $\lambda$) or the simplified one $\tilde \sigma_\lambda$: the joint moments generating functions of spin variables distributed according to both Aldous-Hoover limits are the same. This proves that the joint distributions of both limits are the same. Therefore, they are equivalent.
    
    The conclusion of the Corollary~\ref{cor:asymp_spin_dist} is obtained by re-expressing the fact that the distributional limit of $(\sigma_i^\ell)_{i,\ell\geq1}$ along this convergent subsubsequence is $(\tilde \sigma_\lambda(v_i,x_{i,\ell}))_{i,\ell \geq1}$ (with $(v_i)_{i\geq1}$ and $(x_{i,\ell})_{i,\ell\geq1}$ all i.i.d. uniform in $[0,1]$) in terms of random measures.
\end{proof}

\begin{proof}[Proof of Corollary \ref{cor:asymp_indep_disthm}]
    Let $\{h_i\}_{i\in[k]}$ be as in the statement of the corollary and $(N_j)_{j\ge 1}$ be a subsequence of $N$. Because the spin variables are tight, we have that there exists some subsubsequence $(N_{j_m})_{m\ge 1}$ along which the spins $\sigma_1,\dots,\sigma_k$ (and their replicas) converge jointly in distribution, with asymptotic Aldous-Hoover representation $\sigma(w,u,v,x)$. 

    By the independence of $\sigma(w,u,v,x)$ on $w$ and $u$ (consequence of Corollary~\ref{cor:asymp_spin_dist}), we have that along this subsubsequence
    \begin{equation}
        \E \Big\langle{\prod_{j=1}^k h_j(\sigma_{j})}\Big\rangle - \prod_{j=1}^k \E\Big\langle{h_j(\sigma_{j})}\Big\rangle \xrightarrow{\rm d}  0\,,\label{meanh-hmean}
    \end{equation}
    where the convergence is in distribution over the regularisation $\lambda$. Because $(\mathbb{R},|\cdot|)$ is separable, convergence in distribution of the induced measure of this random variable is equivalent to convergence in Prokhorov's metric ${\rm d_P}(\cdot,\cdot)$. We have then proved that every subsequence has a subsubsequence s.t.
    \begin{equation*}
        \E \Big\langle{\prod_{j=1}^k h_j(\sigma_{j})}\Big\rangle - \prod_{j=1}^k \E \Big\langle{h_j(\sigma_{j})}\Big\rangle \xrightarrow{\rm d_P} 0\,.
    \end{equation*}
    Because the limit is always the same, we have the convergence along the entire sequence $N$. This means that as $N\to+\infty$ \eqref{meanh-hmean} holds. And finally, because all the factors $h_j(\sigma_{j})$ are bounded, this implies convergence in $L^2$.

    %By the independence of $\sigma(w,u,v,x)$ on $w$ and $u$ (consequence of Corollary \ref{cor:asymp_spin_dist}), we have that along this subsubsequence strong decoupling holds (which is seen directly using representation $\tilde \sigma(v,x)$ from Corollary \ref{cor:asymp_spin_dist}):
   %  \begin{equation}
   %      \E \Big\langle{\prod_{j=1}^k h_j(\sigma_{j})}\Big\rangle - \prod_{j=1}^k \E\Big\langle{h_j(\sigma_{j})}\Big\rangle \rightarrow 0\,.\label{meanh-hmean}
   %  \end{equation}
   % We have then proved that every subsequence has a subsubsequence along which strong decoupling holds. Because the limit is always the same, we have the convergence along the entire sequence $N$. This means that as $N\to+\infty$ \eqref{meanh-hmean} holds.
\end{proof}

\section*{Appendix}

\subsection*{Proof of \eqref{pertIsnonPert}}

 Denote $f_N:=\e F_N/N$ and $f_N^{\rm pert}:= \e F_N^{\rm pert}(\lambda)/N$. By the triangle inequality $$|f_N^{\rm pert}(\lambda)-f_N|\le |f_N^{\rm pert}(\lambda_0,(\lambda_k)_{k\ge 1})-f_N^{\rm pert}(\lambda_0=0,(\lambda_k)_{k\ge 1})|+|f_N^{\rm pert}(\lambda_0=0,(\lambda_k)_{k\ge 1})-f_N^{\rm pert}(\lambda=(0))|$$ where $f_N^{\rm pert}(\lambda=(0))=f_N$ is the unperturbed normalised free energy. We know from \eqref{second-derivative-average} that
\begin{align*}
 \Big|\frac{df_N^{\rm pert}(\lambda)}{d\lambda_0} \Big|=\frac{\eps_N}2|\e\la R_{1,2}\ra|\le \frac{\eps_N}2\,.
\end{align*}
Therefore $|f_N^{\rm pert}(\lambda_0,(\lambda_k)_{k\ge 1})-f_N^{\rm pert}(\lambda_0=0,(\lambda_k)_{k\ge 1})|\le \lambda_0\eps_N/2\le \eps_N/2$. 

We now consider the second term. Let $k\ge 1$. We recall \eqref{der_freeEn_bounded} that says $$\Big|\frac{df_N^{\rm pert}(\lambda)}{d\lambda_k} \Big|\le \frac{6s_N}{N}\,.$$
Therefore $|f_N^{\rm pert}(\lambda_0=0,(\lambda_k)_{k\ge 1})-f_N^{\rm pert}(\lambda=(0))|\le (6s_N/N)\sum_{k\ge 1} \lambda_k\le 6s_N/N$ as $\lambda_k\in[2^{-k-1},2^{-k}]$. By hypothesis $\eps_N$ and $s_N/N$ both vanishe as $N$ grows, thus the result.

Another more direct way to see that $|f_N^{\rm pert}(\lambda)-f_N|=o_N(1)$ is to write a bound that is uniform in $\sigma$ for the perturbation $|\mathcal{H}^{\rm gauss}_N(\sigma,\lambda_0) + \mathcal{H}^{\rm exp}_N(\sigma,\lambda)|$ and use it to extract it from $f_N^{\rm pert}(\lambda)$, but this yields a weaker convergence of the order $|f_N^{\rm pert}(\lambda)-f_N|=O(\sqrt{\eps_N}+s_N/N)$.

\subsection*{Proof of inequality~\eqref{remarkable}}
Let $R_{1,*} :=\sigma\cdot \sigma^*/N$. We start by proving the identity
\begin{align}
-2\,\mathbb{E}\big\langle R_{1,*}(\mathcal{L} - \mathbb{E}\langle \mathcal{L}\rangle)\big\rangle
&=\mathbb{E}\big\langle (R_{1,*} - \mathbb{E}\langle R_{1,*} \rangle)^2\big\rangle
+ \mathbb{E}\big\langle (R_{1,*}-   \langle R_{1,*} \rangle)^2\big\rangle\,.\label{47}
\end{align}
Recall $\lambda_{0,N}:= \eps_N\lambda_0$. Using the definition \eqref{def_L} gives
\begin{align}
2\,\mathbb{E}\big\langle R_{1,*} (\mathcal{L} -\mathbb{E}\langle \mathcal{L} \rangle) \big\rangle
   =\, &  \mathbb{E} \Big [ \frac{1}{N}\big\langle R_{1,*} \|\sigma\|^2 \big\rangle - 2 \langle R_{1,*}^2 \rangle - \frac{1}{N\sqrt{\lambda_{0,N}}}\big\langle R_{1,*}\, Z \cdot \sigma \big\rangle \Big ] \nonumber \\
  &  \qquad-  \mathbb{E}\langle R_{1,*} \rangle  \, \mathbb{E} \Big [ \frac{1}{N}\big\langle \|\sigma\|^2 \big\rangle - 2 \langle R_{1,*} \rangle - \frac{1}{N\sqrt{\lambda_{0,N}}} Z \cdot \langle \sigma\rangle \Big ]\,. \label{eq:QL:1}
\end{align}
A gaussian integration by part yields
\begin{align*}
\frac{1}{N\sqrt{\lambda_{0,N}}}\mathbb{E}\big\langle R_{1,*} \,Z \cdot \sigma \big\rangle 
  & = \frac{1}{N}\mathbb{E}\big\langle R_{1,*} \|\sigma\|^2 \big\rangle - \frac{1}{N}\E\big\langle R_{1,*} \,\sigma \cdot \langle \sigma \rangle \big \rangle =\frac1{N}\mathbb{E}\big\langle R_{1,*} \|\sigma\|^2 \big\rangle - \E \langle R_{1,*}  \rangle^2\,.
\end{align*}
Fort the last equality we used the Nishimori identity as follows
$$
\frac1N\E\big\langle R_{1,*}\, \sigma \cdot \langle \sigma \rangle \big \rangle=\frac1{N^2}\E\big\langle (\sigma\cdot \sigma^*) (\sigma \cdot \langle \sigma \rangle) \big \rangle= \frac1{N^2}\E \big\langle (\sigma^*\cdot \sigma) (\sigma^* \cdot \langle \sigma \rangle) \big \rangle=\E \langle R_{1,*} \rangle^2\,.
$$ 
We have already proved $\mathbb{E} \langle Z\cdot  \sigma \rangle/\sqrt{\lambda_{0,N}} 
=  \mathbb{E}\langle \|\sigma\|^2 \rangle - \E \langle R_{1,*}\rangle$ in \eqref{NishiTildeZ}. Therefore \eqref{eq:QL:1} simplifies to 
\begin{align*}
2\,\mathbb{E}\big\langle R_{1,*} (\mathcal{L} -\mathbb{E}\langle \mathcal{L} \rangle) \big\rangle &= \mathbb{E}\langle R_{1,*}\rangle^2 - 2\,\mathbb{E}\langle R_{1,*}^2\rangle +  (\mathbb{E}\langle R_{1,*}\rangle)^2 \nonumber\\
&= -  \big ( \mathbb{E}\langle R_{1,*}^2\rangle - (\mathbb{E}\langle R_{1,*}\rangle)^2 \big ) -  \big ( \mathbb{E}\langle R_{1,*}^2\rangle - \mathbb{E}\langle R_{1,*}\rangle^2 \big )
\end{align*} 
which is identity \eqref{47}. This identity implies the inequality
\begin{align*}
2\big|\mathbb{E}\big\langle R_{1,*}(\mathcal{L} - \mathbb{E}\langle \mathcal{L}\rangle)\big\rangle\big|&=2\big|\mathbb{E}\big\langle (R_{1,*}-\mathbb{E}\langle R_{1,*} \rangle)(\mathcal{L} - \mathbb{E}\langle \mathcal{L}\rangle)\big\rangle\big|\ge \mathbb{E}\big\langle (R_{1,*} - \mathbb{E}\langle R_{1,*} \rangle)^2\big\rangle
\end{align*}
and an application of the Cauchy-Schwarz inequality gives
\begin{align*}
2\big\{\mathbb{E}\big\langle (R_{1,*}-\mathbb{E}\langle R_{1,*} \rangle)^2\big\rangle\, \mathbb{E}\big\langle(\mathcal{L} - \mathbb{E}\langle \mathcal{L}\rangle)^2\big\rangle \big\}^{1/2} \ge\mathbb{E}\big\langle (R_{1,*} - \mathbb{E}\langle R_{1,*} \rangle)^2\big\rangle:= {\rm Var}(R_{1,*})\,.
\end{align*}
Finally using the consequence of the Nishimori identity ${\rm Var}(R_{1,*})={\rm Var}(R_{1,2})$ ends the proof.

\subsection*{Asymptotic multioverlaps in terms of the Aldous-Hoover representation} Let us start by showing that $\la\,\cdot\,\ra$ asymptotically becomes the expectation in the random variables $(u_\ell)$, $(x_{i,\ell})$. Consider a generic joint moment of the quenched Gibbs measure over finitely many spins and replicas, where spins are grouped according to their replica index. Using the Aldous-Hoover representation \eqref{AHrepr} these asymptotically become, in the considered subsequential limit, 
\begin{align*}
 \e \big\la \prod_{\ell \le n} \prod_{i \in \mathcal{C}_\ell} \sigma_i^\ell \big\ra= \e \prod_{\ell \le n} \big\la \prod_{i \in \mathcal{C}_\ell} \sigma_i^\ell \big\ra \to  \e \prod_{\ell \le n}  \prod_{i \in \mathcal{C}_\ell} \sigma(w,u_\ell,v_i,x_{i,\ell})=\e_{w,(v_i)} \prod_{\ell \le n} \e_{u_\ell, (x_{i,\ell})_{i\in \mathcal{C}_\ell}}  \prod_{i \in \mathcal{C}_\ell} \sigma(w,u_\ell,v_i,x_{i,\ell})
\end{align*}
where $w$, $(u_\ell)$, $(v_i)$ and $(x_{i,\ell})$ are i.i.d. uniform $\mathcal{U}[0,1]$ random variables. By identification we get that for a given replica $\sigma^\ell$ the expectation $\la\,\cdot\,\ra$ asymptotically translates into the expectation with respect to $u_\ell$ and $(x_{i,\ell})_{i\in \mathcal{C}_\ell}$, so in general for a function of multiple replicas $\la\,\cdot\,\ra$ becomes the expectation over all the variables indexed by a ``replica index'' $(u_\ell)$, $(x_{i,\ell})$.

Next we prove identity \eqref{asymptMultiover}.
Let us consider a generic (finite) multioverlaps joint moment. Define sets $\{\mathcal{L}_i\}_{i\ge 1}$, whose only finitely many of them are non empty, where $\mathcal{L}_i$ is a finite set of replica indices corresponding to the replicas whose $i$th spin appears in the considered multioverlaps joint moment. Recall that multioverlaps joint moments can be reduced to a product over spins $(\sigma_i^\ell)$ as already observed in \eqref{multi_are_spins}. Let us write multioverlaps as $R_{\mathcal{L}_i}:=\e_{i_j} \prod_{\ell\in\mathcal{L}_i}\sigma^\ell_{i_j}$ where $i_j$ is uniform among $\{1,\ldots,N\}$. Defining an empty product to be one $\prod_{\emptyset}(\cdots):=1$, a generic multioverlaps joint moment reads
\begin{align*}
  \e\big\la \prod_{i\ge 1} R_{\mathcal{L}_i}\big\ra = \e\big\la \prod_{i\ge 1}\e_{i_j}\prod_{\ell\in \mathcal{L}_i}\sigma_{i_j}^{\ell}\big \ra= \e\big\la \prod_{i\ge 1}\prod_{\ell\in \mathcal{L}_i}\sigma_{i_j}^{\ell}\big \ra +\mathcal{O}(N^{-1})
\end{align*}
where the last equality from the symmetry among spins \eqref{spinSymm}. Then
\begin{align*}
  \e\big\la \prod_{i\ge 1}\prod_{\ell\in \mathcal{L}_i}\sigma_{i_j}^{\ell}\big \ra\to \e_{w,(u_\ell)} \prod_{i\ge 1} \e_{v_{i}} \prod_{\ell \in\mathcal{L}_i}\e_{x_{i,\ell}} \sigma(w,u_\ell,v_{i},x_{i,\ell})=\e_{w,(u_\ell)} \prod_{i\ge 1} \int_0^1 \prod_{\ell \in\mathcal{L}_i}\bs(w,u_\ell,v)\,dv\,.
\end{align*}
By identification we obtain the claimed identity \eqref{asymptMultiover}.

\section*{Acknowledgements}
J.B. is extremely grateful to Nicolas Macris and Chun-Lam Chan for numerous discussions. D.P. was partially supported by NSERC and Simons Fellowship.

\bibliographystyle{plain}  
\bibliography{refs}

\end{document}